\documentclass[11pt,reqno]{amsart}
\usepackage[margin=2cm]{geometry}
\usepackage{graphicx} 
\usepackage{float}
\usepackage{caption}
\usepackage{subcaption}

\newtheorem{theorem}{Theorem}[section]

\newtheorem{proposition}{Proposition}[section]

\usepackage{enumitem}
\usepackage{xcolor}

\allowdisplaybreaks

\usepackage{amsmath}
\usepackage{empheq}
\usepackage{comment}
 \usepackage{amsaddr}

\usepackage{appendix}
\usepackage{algorithm}
\usepackage{algpseudocode}

\usepackage{tabularx}

\usepackage[backend=bibtex,
style = numeric,
]{biblatex}
\newcommand\bX{\mathbf X}
\newcommand\bW{\mathbf W}

\newcommand\bm{\mathbf m}

\newcommand\cP{\mathcal P}

\def\balpha{\boldsymbol{\alpha}}

\def\RR{\mathbb R}

\usepackage{cancel}

\title{Leveraging the turnpike effect for Mean Field Games numerics}
\author{René Carmona, Claire Zeng}
\address[A1,A2]{Princeton University, Princeton, USA}
\email[A1,A2]{rcarmona@princeton.edu, cszeng@princeton.edu}
\date{\today}

\begin{document}

\maketitle

\begin{abstract}
    Recently, a deep-learning algorithm referred to as Deep Galerkin Method (DGM), has gained a lot of attention among those trying to solve numerically Mean Field Games with finite horizon, even if the performance seems to be decreasing significantly with increasing horizon. On the other hand, it has been proven that some specific classes of Mean Field Games enjoy some form of the turnpike property identified over seven decades ago by economists. The gist of this phenomenon is a proof that the solution of an optimal control problem over a long time interval spends most of its time near the stationary solution of the ergodic solution of the corresponding infinite horizon optimization problem. After reviewing the implementation of DGM for finite horizon Mean Field Games, we introduce a ``turnpike-accelerated'' version that  incorporates the turnpike estimates in the loss function to be optimized, and we perform a comparative numerical analysis to show the advantages of this accelerated version over the baseline DGM algorithm. We demonstrate on some of the Mean Field Game models with local-couplings known to have the turnpike property, as well as a  new class of linear-quadratic models for which we derive explicit turnpike estimates.
\end{abstract}


%%%%%%%%%%%%%%%%%%%%%%%%%%%%%%%%%%%%%%%%%%%%%%%%
\section{\textbf{Introduction and Relevant Literature} }

In stochastic control, and more recently in dynamic game theory, the turnpike phenomenon refers to a modern macro-economics principle, that was first identified by Dorfman, Samuelson, and Solow in \cite{dorfman_linear_1987}. It states that in some optimal control problems  in large time, the optimal trajectory between two points that are far apart can often be split into three segments: 
\begin{itemize}
    \item a first short-time arc departing from the point of origin or initial condition to a steady-state, 
    \item a second long-time arc that stays close to the steady-state for most of the time (this is the so-called turnpike), 
    \item a third short-time arc that leaves the turnpike at some time that is close to the time horizon and manages to satisfy the terminal condition. 
\end{itemize}
When it holds, the turnpike property is useful because it describes qualitatively the long time behavior of the finite horizon solution, by leveraging information from the solution of a (usually simpler) stationary problem without having to analytically solve the original problem. This phenomenon has been identified in one form or another, in many macro-economic models, and more recently, it appeared in the optimal control literature and even more recently in the mean field game literature. For deterministic linear optimal control problems, \cite{trelat_steady-state_2018} shows that the control and the state trajectories remain exponentially close to the solutions of the steady-state version of the problem in finite dimension. Extensions are treated for non-linear problems in \cite{trelat_turnpike_2015} under some appropriate controllability and smallness assumptions, and in \cite{porretta_turnpike_2018} for the linear-quadratic case in finite dimension under the Kalman controllability rank condition. 

The long time behavior of mean field games (MFG) has been studied in several settings.  This behavior has been identified in \cite{ferreira_convergence_2014} and \cite{gomes_discrete_2010} for finite state mean field games, and in \cite{porretta_long_2012} and \cite{porretta_long_2013} for the continuous case. More specifically, the authors of these papers compare the solution of a finite horizon MFG to its stationary analog, as initially developed by Lasry-Lions in \cite{lasry_jeux_2006_a} and \cite{lasry_jeux_2006_b}. The long time behavior of the value function has been investigated on the torus in the case of local couplings for a quadratic Hamiltonian of the form $H(x,p) = \frac{1}{2} \lvert p \rvert^2$ in \cite{porretta_long_2012,porretta_long_2013,cardialaguet-nonlocal}, and later on, for more general Hamiltonians \cite{porretta_turnpike_2018}. Under a strong monotonicity assumption on the running cost coupling, the turnpike estimates hold with an exponential rate. The cases of local and non-local couplings have been addressed respectively in \cite{cardialaguet-nonlocal}     
and \cite{porretta_long_2012}. Under the additional assumption of smoothing properties of the coupling functions, an exponential rate of convergence can be proven. A more precise analysis of the type of convergence and relationship between the ergodic and finite horizon problems has been performed in \cite{cardaliaguet_long_2019} through the study of the long time behavior of the master equation. The paper offers a description of the relationship at the limit of the time horizon $T \rightarrow \infty$ between the finite time value function, the ergodic value function, the ergodic constant and time. In fact, the monotonicity assumption can be somewhat relaxed as demonstrated in \cite{cirant_long_2021}, where the exponential convergence is still ensured by the presence of volatility, which compensates for the lack of strict monotonicity. The previous results are of a global nature, local results with exponential rates are also available provided that the value function and the mean field are close enough to the stationary solution. Illustrations of the manifestation of the global turnpike property for mean field games are available for a congestion model in \cite{achdou_mean_2020} and an optimal execution problem in \cite{OE_anon_id}. A local turnpike property has been shown for a Kuramoto mean field game in \cite{Kuramoto}. 

The present paper is concerned with numerical methods to solve mean field games. In the existing literature, the first numerical schemes consisted in the search of a fixed point by iteration of the solutions of discretized versions of the Hamilton-Jacobi-Bellman (HJB) and the Kolmogorov-Focker-Planck (KFP)  equations : finite-difference algorithms are one example. Fully-implicit monotone versions of the finite difference scheme enabled considering long-time horizons at the price of using too many computing resources when the state space dimension $d \geq 4$.  More recently, deep learning methods have been developed to solve such problems. The Deep Galerkin Method (DGM) can be used to solve the MFG Partial Differential Equation (PDE) system by parameterizing the value function $u$ and the probability density $m$ as two neural networks that are then trained by Stochastic Gradient Descent (SGD) to minimize a loss function involving the residuals of the PDEs, the initial and terminal conditions, as well as eventual boundary or periodicity constraints.  An advantage of the DGM algorithm is that it is \textit{mesh-free} by the random sampling of time and space points to form batches : this therefore enables the method to be scalable for systems of high dimension.  Another alternative deep learning algorithm focuses on the MFG Forward Backward Stochastic Differential Equations (FBSDE) associated and uses a shooting trick \textit{à la Sannikov} by replacing the terminal condition of the Backward Stochastic Differential Equation (BSDE) initial condition to be determined by SGD. 

\vskip 2pt
The turnpike property has been touted as a potential tool to improve existing numerical computations for deterministic control problems. For example, \cite{trelat_steady-state_2018} suggests the following scheme that relies on the intuitive idea that at time $t = T/2$, the solution is close to the steady-state solution, and the complete path can be recovered by a double-sided shooting method on the system derived from the Pontryagin Maximum Principle: 
\begin{itemize}
    \item starting at time $T/2$ from the steady state solution, reverse the direction of time in the dynamics of the state, and solve the problem by a shooting algorithm $0 \leftarrow T/2$  to match the initial condition; 
    \item similarly following the natural flow of time, solve the problem from $T/2$ to the time horizon $T$ and use again a shooting method $T/2 \rightarrow T$ to match the terminal condition.
\end{itemize}
However, there is no straightforward generalization of this idea to Mean Field Games of second order. 
The main obstacle is the fact that the Pontryagin forward-backward system is now stochastic, deriving time-reversed dynamics is typically difficult, and when compounded with a shooting method, the results can be very poor, especially over long time intervals.
At best, the Mean Field Games turnpike estimates give us an information about how close the state statistical distribution can be to the ergodic one, and possibly the proximity of the derivative of the value function $Du(t,x)$ to the ergodic derivative $D \overline{u}$. Here $Du(t,x)$ is the decoupling field of the FBSDE system derived from the Pontryagin Stochastic Maximum Principle. However, trying to time-reverse an SDE (corresponding to the dynamics of the BSDE) to generalize the above scheme is no easy task because of the forward direction of flow of information and the fact that solutions to stochastic equations need to remain non-anticipative. The first result about time reversion of diffusion equation models is due to \cite{ANDERSON1982313}: constructing the reverse-time representation of a solution of an SDE requires the knowledge of the time-dependent score function $\nabla_x \log p_t$ that describes the evolution of the probability density $p(X_t, t | X_s, s)$ of the state process $X = (X_t)_{t >0}$. Estimating the score function of a generic SDE can be achieved through training a score-based model on samples with score matching as in \cite{song2021scorebased}. To the best of our knowledge, there are no known result for stochastic control problems, where the forward SDE dynamics are \textit{a priori} unknown until the control is determined.  

The closest known application of the turnpike idea for Mean Field Games is suggested in \cite{achdou_mean_2020}: since the Newton algorithm used in the finite difference method requires a good initial condition, sometimes it is useful to use the stationary solution of the ergodic problem as an initializer of this step. However, this warm-start is suggested not only for problems with a turnpike property but for any problem provided that a stationary solution is relatively easy to compute. More precisely, this trick enables to have a good initial guess for the distribution variable (which is crucial to guarantee the positivity of the distribution approximations in the next iterations). Other alternatives are to either simulate the problem on a coarser grid and use it as the initial guess or to use a continuation method with respect to the viscosity parameter by decreasing it to the desired value. 

\vskip 2pt
In this work, we focus on the numerical solutions of MFGs over long time horizons when the shooting methods do not perform well, and for second order MFGs for which time reversal of stochastic dynamics is an issue. As demonstrated in the recent literature on the subject, we rely on machine learning algorithms to compute numerical approximations of the solutions. To be specific, we work with
the Deep Galerkin Method (DGM) algorithm which has gained a lot of interest thanks to its flexibility and generality. However, to resolve the weaknesses of the method for long time horizons, we propose an improved version of the algorithm based on the turnpike property. This is achieved by incorporating forms of the turnpike estimates into the loss function without any modification of the Stochastic Gradient Descent (SGD) update rules. To demonstrate the positive impact of this modified objective function in the SGD, we  perform a comparison study between the baseline DGM algorithm and our \textit{turnpike-accelerated} version. We  focus on models for which we have explicit exponential rates of convergence to objectively assess the nature of the improvements. To further the case for the improved algorithm, we introduce a new class of linear quadratic Mean Field Games for which we derive global explicit exponential turnpike estimates. To the best of our knowledge, it is the first of its kind.  Indeed, we are only aware of \cite{ARAPOSTATHIS2017205} for the long time behavior of MFGs set on unbounded domains under the assumption of geometric ergodicity for the dynamics and where no rate of convergence is derived.  Comparative results from the implementations of the two DGM algorithms document the merits of the turnpike-accelerated computations.

The rest of the paper is organized as follows. Section \ref{section:2} reviews a class of Mean Field Game models  with local coupling interactions, and the turnpike estimates available for them in the literature. We also identify a specific model which can be solved explicitly and we illustrate graphically the potential impact of the turnpike property on possible numerics. Section \ref{section:3} describes the two Deep Galerkin Methods (DGM) that are being compared in this work. In particular, we detail the form of the accelerated DGM which we propose to leverage the turnpike property. Section \ref{section:4} introduces a new class of explicitly solvable models for which we can derive explicit turnpike estimates and test numerically the performances of both versions of the DGM algorithm introduced in Section \ref{section:3}. Finally, Section \ref{section:5} concludes.

%%%%%%%%%%%%%%%%%%%%%%%%%%%%%%%%%%%%%%%%%%%%%%%%%%%%%%%%%%%%
\section{\textbf{Mean Field Games PDE Systems and Turnpike Estimates} }
\label{section:2}

This section first introduces the finite horizon and ergodic mean field game models that are considered in the rest of this paper. We refer the reader to the original papers \cite{lasry_jeux_2006_a} and \cite{lasry_jeux_2006_b} for more details. We then present a brief overview of the turnpike estimates linking those models, and focus on the typical case of local couplings as studied in \cite{porretta_long_2012}.

\vskip 2pt
For a given dimension $d \in \mathbb{N}^* $, we denote the $d$-dimensional flat torus  by $\mathbb{T}^d = \mathbb{R}^d \slash \mathbb{Z}^d $. This will be the state space of our stochastic systems. We denote the unit cube by $Q = [0,1]^d$. 

%%%%%%%%%%%%%%%%%%%%%%%%%%%%%%%%%%%%%%%%%%
\subsection{Finite Horizon Mean Field Game}

The most common setup of finite horizon Mean Field Games is given by the following system of Partial Differential Equations (PDE): 
 
 \begin{subequations}
    \begin{empheq}[left={(MFG-T):= \empheqlbrace}]{align}
& - \partial_t u - \kappa \Delta u + H(x, Du) =  F(x, m) &\text{ in } (0, T) \times \mathbb{T}^d, \\
& \partial_t m - \kappa \Delta m - \operatorname{div}(m H_p(x, Du)) = 0 &\text{ in } (0, T) \times \mathbb{T}^d,  \\
& m(0, \cdot) = m_0 , \, u(T,x) = G(x, m(T)). 
    \end{empheq}
    \label{eq:mfg_t}
    \end{subequations}
with additional periodic boundary conditions since we are working on the torus. 

The heuristic interpretation of this PDE system is associated with a representative agent that controls the time evolution of their state $\bX = (X_t)_{0\le t\le T}$ in $\mathbb{R}^d$ through a control process $\balpha = (\alpha_t)_{0\le t\le T}$. For the purpose of this paper, we work with controls in feedback form, so $\alpha_t=\alpha_t(X_t)$ for some deterministic (feedback) function $x\mapsto \alpha_t(x)\in\RR^d$. The state of the agent is typically governed by the following SDE: 
\begin{equation}
    dX_t = \alpha_t(X_t)  dt + \sigma dW_t,\qquad  X_0 = \xi \sim m_0,
\end{equation}
where $(W_t)_{t\ge 0}$ is a standard $d$-dimensional Brownian motion. For a fixed probability flow $\bm=(m_t)_{0\le t\le T}$ that represents the distribution of the states of the other players, the representative agent aims at minimizing an objective function, expressed as: 

\begin{equation*}
    J^T(\balpha | \bm) = \mathbb{E} \Big[ \int_0^T \Bigl(\frac{1}{2} \lvert \alpha_s \rvert^2 + F(X_s, m_s)\Bigr) ds + G(X_T, m_T)\Big],
\end{equation*}
where the total cost consists of a running cost that is accumulated over time and a terminal cost that penalizes the state in which the agent lands at the end of the time horizon $T>0$. Different existence and uniqueness results are available for this PDE system depending on the nature of the coupling, whether they are smoothing or local couplings. Typically, the coupling is provided by an instantaneous cost function $F$ from $\RR^d\times\cP(\RR^d)$ into $\RR$ where $\cP(\RR^d)$ denotes the space of probability measures on $\RR^d$.
We say that the coupling is local when $F$ is of the form $F(x,m)=\tilde F(x,m(x))$ for a function $\tilde F$ (which we shall still denote by $F$) from $\RR^d\times\RR$ into $\RR$, the probability measure $m$ being assumed to have a density whose value at $x$ we denote by $m(x)$ by a mild abuse of notation.

In any case, because of the special form of the state dynamics and of the running cost, $H(x,p) = \frac{1}{2} \lvert p \rvert^2$ and we expect that the optimal control will be given in feedback form by $\alpha^*(t,x) = - Du^T(t,x) $. 

%%%%%%%%%%%%%%%%%%%%%%%%%%%%%%%%%%%%
\subsection{Ergodic Mean Field Game}
 The case of ergodic cost function offers an infinite horizon model (i.e. $T=\infty$) for which the players are infinitely patient. Given the form of the cost given below in \eqref{eq:ergodic_cost}, the solution of the ergodic problem is independent of the behavior of the players and the state of the system at any finite time, and relies on a value function $u_t=\bar u$, marginal distributions $m_t=\bar m$ and an optimal feedback control function $\alpha_t=\bar\alpha$ which are independent of time and solve the PDE system:

 \begin{subequations}
    \begin{empheq}[left={(MFG-e):= \empheqlbrace}]{align}
& \lambda - \kappa \Delta \overline{u}+ H(x, D\overline{u}) =  F(x, \overline{m}) &\text{ in }   \mathbb{T}^d, \\
& - \kappa \Delta \overline{m} - \operatorname{div}(\overline{m} H_p(x, D\overline{u})) = 0 &\text{ in }  \mathbb{T}^d,  \\
& m \geq 0, \int_{\mathbb{T}^d} \overline{m}(x)dx = 1, \int_{\mathbb{T}^d} \overline{u}(x) dx = 0. 
    \end{empheq}
    \label{eq:mfg_ergodic}
    \end{subequations}    
If the coupling $F(x, \cdot)$ is monotone in the sense of \cite{lasry_jeux_2006_a}, then the ergodic problem is well-posed; that is, there exists a unique couple $(\overline{u},\overline{m})$ and a unique constant $\overline{\lambda}$ which solve the above system. Moreover, $\overline{u}, \overline{m}$ are smooth and $\overline{m}>0$. The heuristic interpretation of the ergodic MFG PDE system is similar to the finite horizon setting, except for the fact that the cost function now reads: 
\begin{equation}
    \label{eq:ergodic_cost}
  \underset{T \rightarrow +\infty}{\lim \inf} \,  \frac{1}{T}  \mathbb{E} \Big[ \int_0^T \frac{1}{2} \lvert \bar\alpha(X_s) \rvert^2 + F(X_s, \bar m) ds \Big].
\end{equation}

The convergence of the solution of the finite horizon MFG towards the stationary ergodic system when $T\to\infty$ was analyzed in several settings. \cite{porretta_long_2012} and \cite{cardialaguet-nonlocal} study this convergence when the Hamiltonian is purely quadratic $H(x,p) = \frac{1}{2} \lvert p \rvert^2$. \cite{porretta_turnpike_2018} establishes a complete description of the long time convergence in the case of non-local smoothing couplings and a uniformly convex Hamiltonian by analyzing the master equation, while \cite{porretta_long_2012} focuses on the case of local couplings and a globally Lipschitz Hamiltonian. This is the setting we will focus on in this paper. For convenience, we provide an overview of the turnpike estimates of \cite{porretta_long_2012} in the next subsection. 

%%%%%%%%%%%%%%%%%%%%%%%%%%%%%%%%%%%%%%%%%%%%%%%%%%%%%%%%%%%%%%%%%%%%%%%%%%%%%%%%%%%%%%%%%%%%%%%%%%%%%%%%%%%%%%%%%
\subsection{Reference Case: $H(x, p) = \frac{1}{2} \lvert p\rvert^2 $ with smooth initial data $m_0, u_T$, $m_0 > 0$} 

The turnpike property for Mean Field Games can take different forms, but they all capture the idea that the solutions become nearly stationary around the ergodic solution for a long time. Weak versions of the turnpike property involve a time-average convergence, while stronger versions provide exponential rates of convergence. 

\vskip 1 \baselineskip

\paragraph{\textbf{Assumptions}}

We assume that $F : \mathbb{R}^d \times \mathbb{R} \ni (x,m)\mapsto F(x,m)\in  \mathbb{R}$ is a $\mathcal{C}^1$ function which is $\mathbb{Z}^d$-periodic in $x$ and increasing in $m$. We also assume that $m_0 : \mathbb{R}^N \rightarrow \mathbb{R}$ is Lipschitz continuous, $\mathbb{Z}^d$-periodic with $m_0 > 0$ and $\int_Q m_0 = 1$, and that $G : \mathbb{R}^d \rightarrow \mathbb{R}$ is  $\mathbb{Z}^d$-periodic and $\mathcal{C}^2$. 

\begin{proposition}[Proposition 1 in \cite{porretta_long_2012} ]
Let the above assumptions hold. 
\begin{enumerate}[label = (\roman*)]
    \item There is a unique classical solution $(u^T, m^T)$ to \eqref{eq:mfg_t}. 
    \item There is a unique classical solution $(\lambda, \overline{u}, \overline{m})$ to \eqref{eq:mfg_ergodic}. Moreover, $\overline{m} =e^{-\overline{u}}/ \int_{Q} e^{-\overline{u}}$.
\end{enumerate}
\end{proposition}

\begin{theorem}[Theorem 2.1 in \cite{porretta_long_2012} ] Let us denote $v^T(t,x) := u^T(tT, x)$ and $\mu^T(t,x) := m^T(tT,x)$. As $T \rightarrow +\infty$, we have:
\begin{enumerate}[label = (\roman*)]
    \item the function $v^T(t, \cdot)/T $ converges to $(1-t) \lambda$ in $L^1(Q)$ uniformly with respect to $t \in [0,1]$; 
     \item the functions $v/T$ and $v^T - \int_{Q} v^T(t,y) dy$ converge respectively to $(1-t) \lambda$ and $\overline{u}$ in $L^2((0,1) \times Q)$; 
     \item the function $\mu^T$ converges to $\overline{m}$ in $L^p((0,1) \times Q)$ for any $p < (d+2)/d$. 
\end{enumerate}
\end{theorem}

\begin{proposition}[Lemma 2.4 and (13) in  \cite{porretta_long_2012} ]

 $$\underset{T \rightarrow+\infty}{\lim} \frac{1}{T} \int_0^T \int_{Q}\lvert Du - D \overline{u} \rvert^2 dx dt = 0,$$
 
    $$\underset{T \rightarrow+\infty}{\lim} \frac{1}{T} \int_0^T \int_{Q} (F(x,m^T) -F(x,\overline{m}) ) (m^T-\overline{m}) dx dt = 0.$$
\end{proposition}

\begin{theorem}
   Furthermore, if we assume the following strengthened form of the Lasry-Lions monotonicity condition, namely there exists $\gamma > 0$ such that for all  $m_1 \geq  m_2$ in $\mathbb{R}$ and all $x \in \mathbb{R}^d$ 
   \begin{equation}
       \label{fo:strong_monotonicity}
   \Bigl(F(x,m_1) - F(x, m_2) \Bigr) (m_1 - m_2) \geq \gamma (m_1 - m_2)^2 
   \end{equation}
   then, there exists $\omega>0$ (independent of $T$) such that: 
    \begin{equation}
        \lVert \tilde{u}(t) - \overline{u} \rVert_{L^1(Q)} \leq \frac{C}{T-t} \Big(e^{-\omega t} + e^{-\omega (T-t)} \Big), \quad\quad t \in (0, T), \label{eq:lin_tp_u}
    \end{equation}
     \begin{equation}
        \lVert m(t) - \overline{m} \rVert_{L^1(Q)} \leq \frac{C}{t}  \Big(e^{-\omega t} + e^{-\omega (T-t)} \Big), \quad\quad t \in (0, T), \label{eq:lin_tp_m}
    \end{equation}
    and 
    \begin{equation}
        \Big\lVert \frac{u^T(t)}{T} - \overline{\lambda} \Big(1 - \frac{t}{T} \Big) \Big\rVert_{L^1(Q)} \leq \frac{C}{T}, \quad\quad t \in (0, T-1),
    \end{equation}        
    where $\tilde{u} = u^T - \langle u^T \rangle $,  $\langle u \rangle = \frac{1}{\lvert Q\rvert }\int_Q u(x)dx$ and $C$ may depend on the initial and terminal conditions.
\end{theorem}

Going over the proof of the exponential turnpike property in \cite{porretta_long_2012}, we can refine the conclusion of this theorem and obtain the explicit exponential rate of convergence: 
\begin{equation}
    \omega = \frac{1}{2} \min \{ 2 \pi^2 \cdot  \underset{Q}{\min} \, \overline{m}, \gamma \},
\end{equation}
stemming from the Poincaré-Wirtinger inequality and a better control of the constants appearing through the proof of Lemma 3.2 in \cite{porretta_long_2012}. 

\vskip 2pt
\paragraph{\textbf{Example}} A typical coupling satisfying the strengthened monotonicity condition is: 
$$
F(x, m) = \gamma m + U(x),
$$ 
for some $\gamma > 0 $ and  a $\mathcal{C}^1 $ function $U$ which is $\mathbb{Z}^d$-periodic. In this setting, players have some preference for their given state $x$ according to the potential $U(x)$, but also care about the number 
of other players in their close vicinity through the linear dependence in $m$. The positive parameter  $\gamma$ corresponds to attractive interactions such as herding. A negative value would correspond to aversion of crowded regions. 

%%%%%%%%%%%%%%%%%%%%%%%%%%%%%%%%%%%%%%%%%%%%%%%%%%%%%%%%%%%%%
\subsection{Numerical Illustration of the Turnpike Phenomenon}
Inspired by the example provided in \cite{carmona_convergence_2021}, we consider a running cost satisfying the previous assumptions with $\gamma = 1$: 
$$ F(x,m) = \gamma m(x) +  50 * (0.1* \cos(2\pi x) + \cos(4 \pi x) + 0.1*\sin(2 \pi (x-\pi/8))).$$
We consider a terminal cost of the shape $G(x,m) = \psi \sin(2\pi (x + \frac{1}{4}))$ for $\psi = 1$ and an initial condition $m_0(x) = \exp(-(x - \mu_0)/(2 \sigma_0^2)) / \int_{Q} \exp(-(y - \mu_0)/(2 \sigma_0^2)) dy $ with $\mu_0 = 0.5$ and $\sigma_0 = 0.2$. 
\begin{figure}[ht]
    \centering
    \begin{subfigure}[t]{0.45\textwidth}
        \centering
        \includegraphics[width = 7cm]{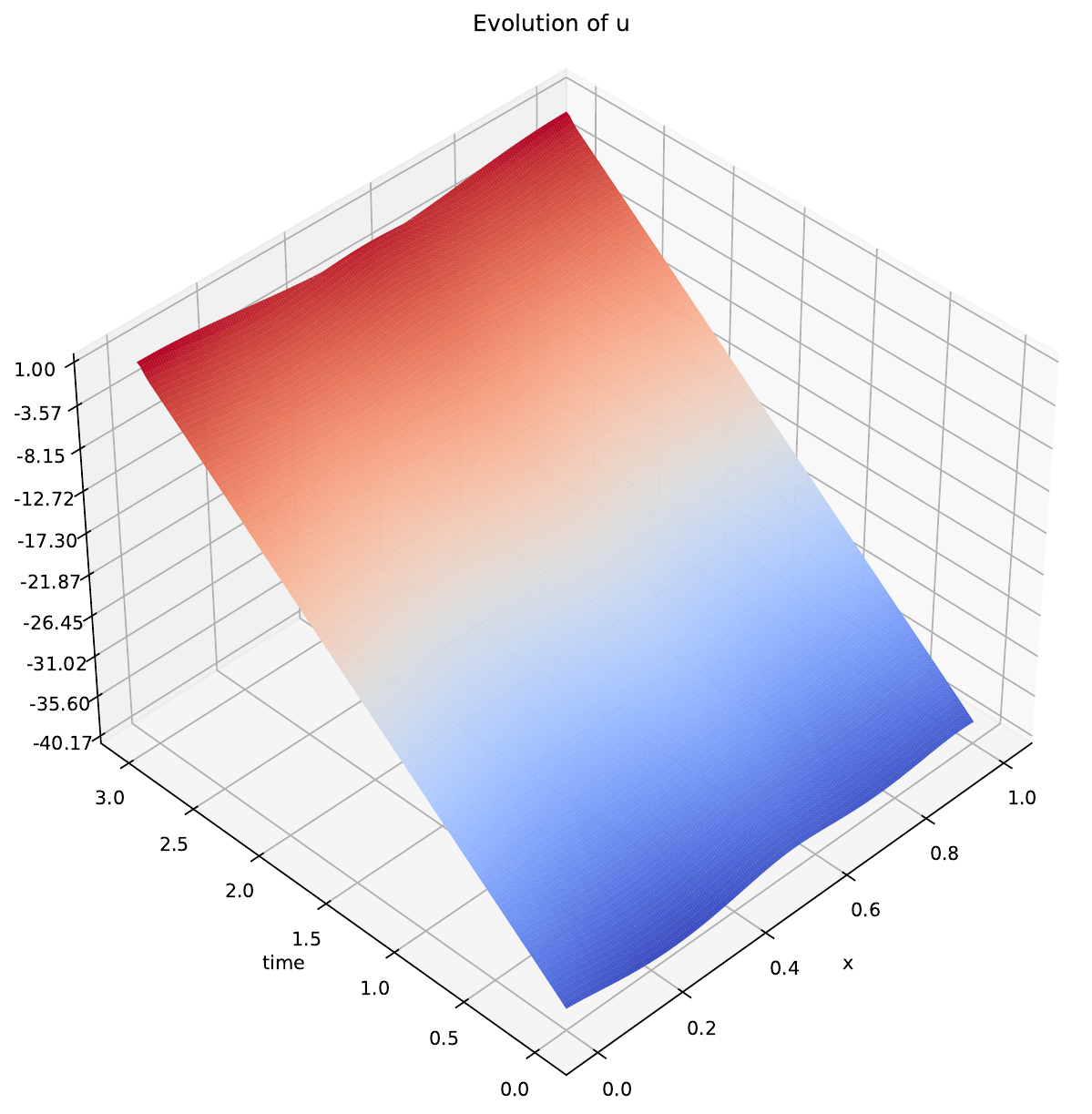}
        \caption{Plot of value function $u^T$}
            \label{fig:agent_trading_speeds_set_0}
    \end{subfigure} 
    ~ 
    \begin{subfigure}[t]{0.45\textwidth}
        \centering
          \includegraphics[width = 7cm]{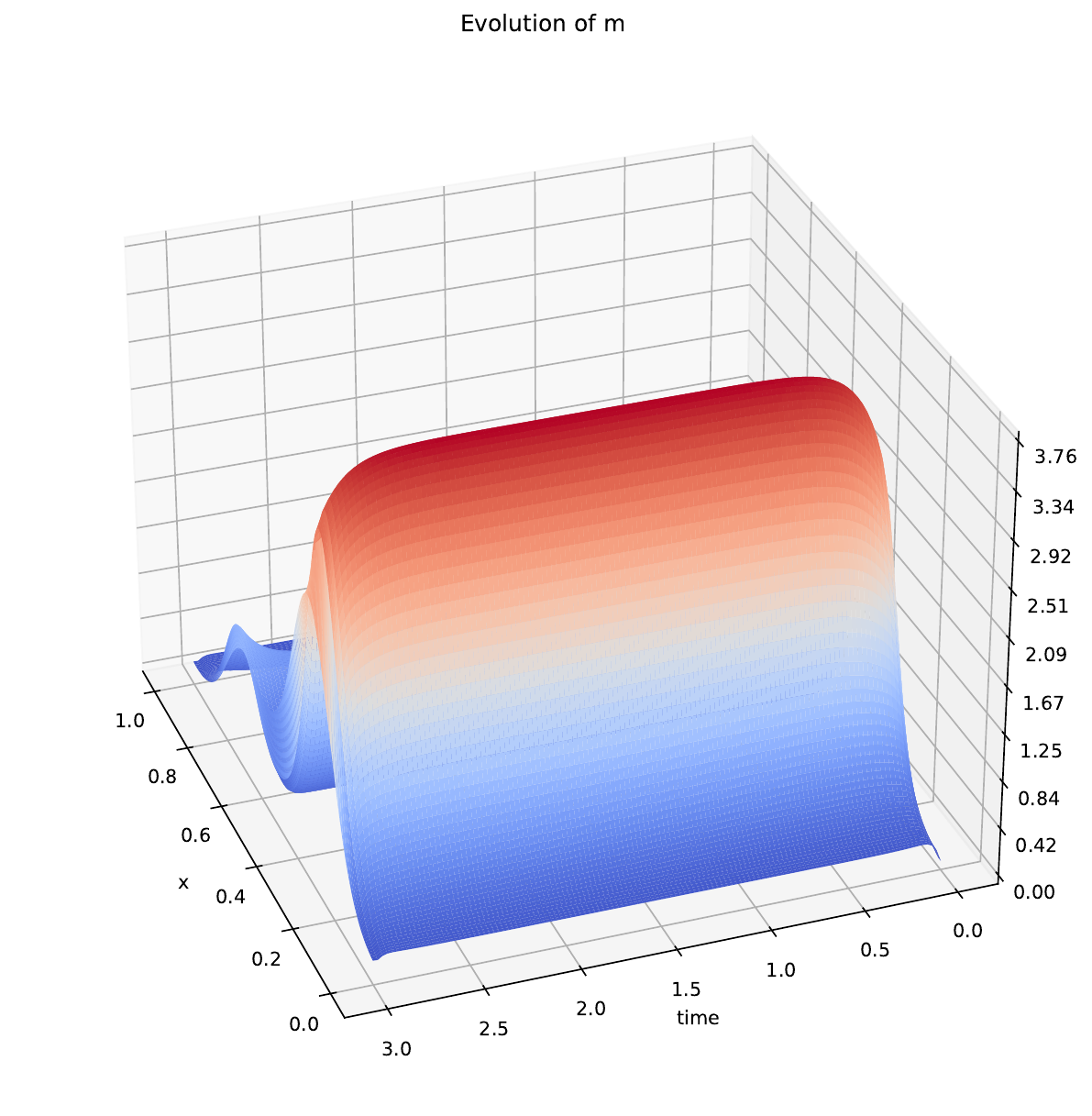}
    \caption{Plot of the density function $m^T$}
    \label{fig:agent_inventories_set_0}
    \end{subfigure}
    \caption{Behavior of the solution $(u^T; m^T)$ of the finite horizon game}
\end{figure}

\begin{figure}[ht]
    \centering
    \begin{subfigure}[t]{0.45\textwidth}
        \centering
        \includegraphics[width = 7cm]{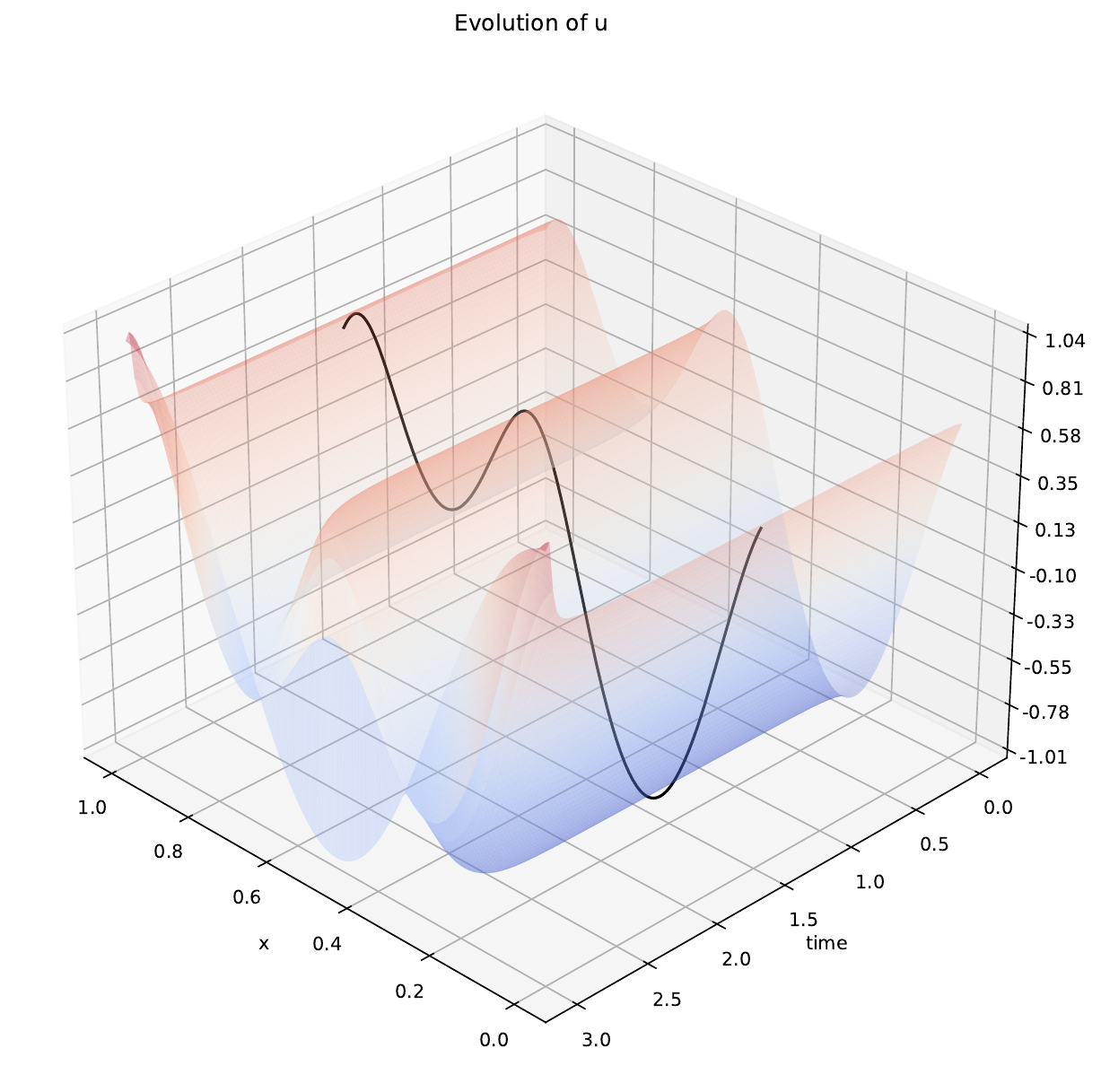}
        \caption{Plot of normalized value function $u^T - \langle u^T \rangle $}
    \end{subfigure} 
    ~ 
    \begin{subfigure}[t]{0.45\textwidth}
        \centering
          \includegraphics[width = 7cm]{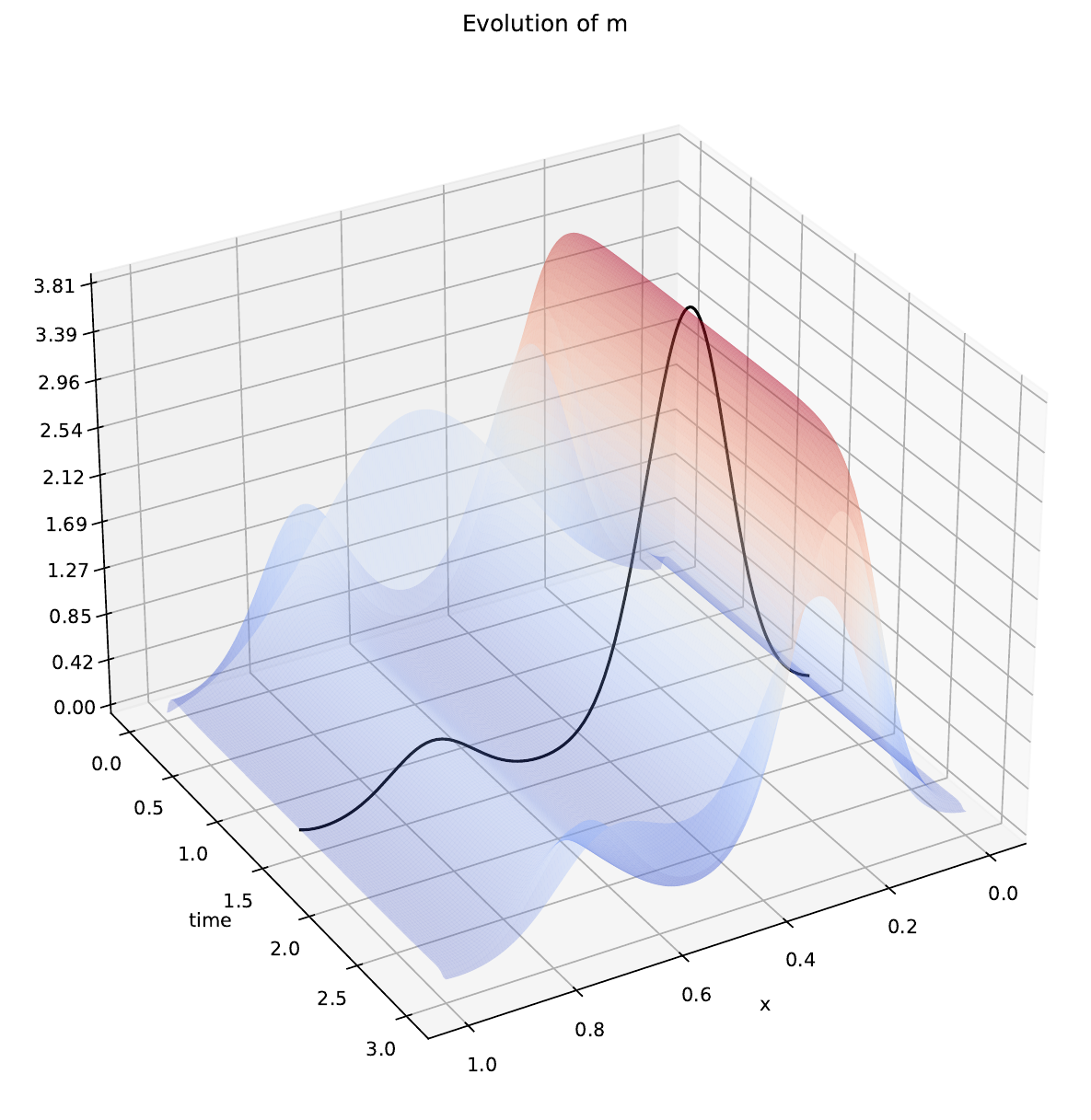}
    \caption{Plot of the density function $m^T$}
    \end{subfigure}
    \caption{Behavior of the solution $(u^T; m^T)$ of the finite horizon game}
\end{figure}

We computed the solution of the finite horizon MFG by a standard Finite Difference Method (FDM): the details of the implementation are provided in Appendix \ref{appendix:A}. The solutions are plotted over a $200\times 200$ grid on $[0,T] \times [0,1]$. We also computed the so-called turnpike by approximating the solution of the ergodic Mean Field Game \eqref{eq:mfg_ergodic} using the Deep Galerkin Method described in \cite{carmona_convergence_2021}. For the sake of completeness, a brief description of the DGM method for the ergodic case is provided in Appendix \ref{appendix:B}. We confirmed that this example has the turnpike property by plotting the $L^2$ norm of $u(t,\cdot) - \int_{Q} u(t,x) dx - \overline{u}$ and $m(t,\cdot) - \overline{m}$  with respect to time in Figure \ref{fig:tp_check}. We observe that starting from the initial density, the distance of the mass to the ergodic distribution quickly converges to zero and stays there for most of the time interval $[0,T]$, before increasing again near the time horizon $T$, when the terminal density is influenced by the terminal cost through the value function. Similarly, we observe that the space-mean adjusted value function stays close to the steady state for most of the time interval, except at times close to the horizon $T$ because of the terminal condition. 

\begin{figure*}[ht]
    \centering
    \includegraphics[width = 8cm]{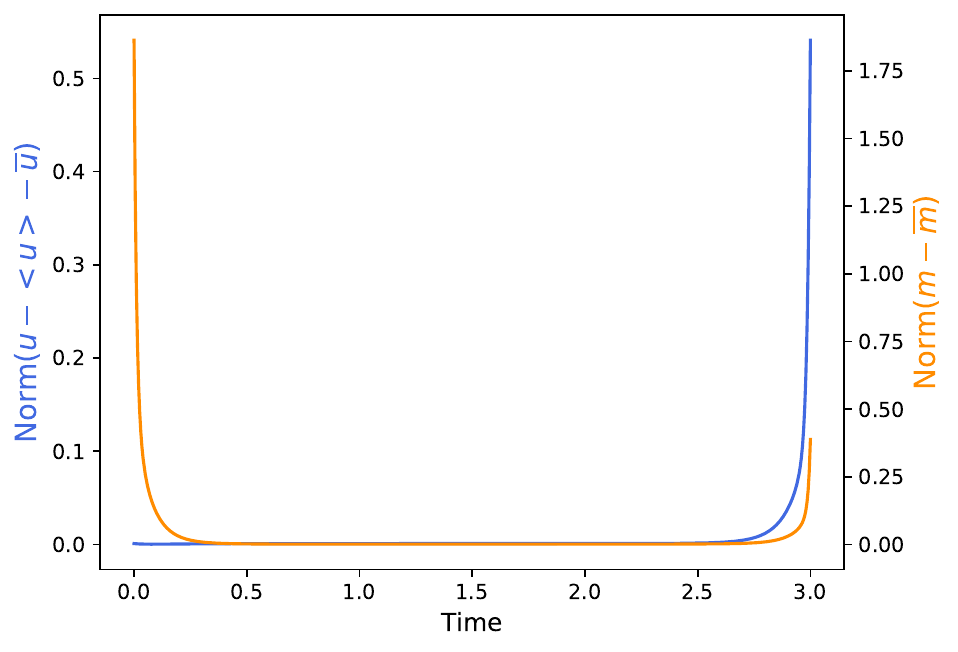}
    \caption{Norm of $m^T - \overline{m}$ and $u^T-\langle u^T \rangle - \overline{u}$}
    \label{fig:tp_check}
\end{figure*} 

%%%%%%%%%%%%%%%%%%%%%%%%%%%%%%%%%%%%%%%%%%%%%%%%%%%%%%%%%%%%%%%%%%%%%%%%%%
\section{\textbf{Deep Galerkin Method: Baseline and Turnpike-accelerated Versions}}
 \label{section:3}

In this section, we provide a description of the two numerical methods that are being investigated : the baseline Deep Galerkin Method as introduced in \cite{SIRIGNANO20181339} and a so-called \textit{turnpike-accelerated} DGM version that incorporates the additional turnpike estimates in the loss function to be minimized. 

%%%%%%%%%%%%%%%%%%%%%%%%%%%%%%%%%%%%%%%%%%%%%%%%%%%%%%%%%%%%%%%%%%%%%%%%%%%
\subsection{Baseline Deep Galerkin Method for the Finite Horizon PDE System}

We approximate the functions $u(t,x)$ and $m(t,x)$ by two deep neural networks $u_{\theta} := \phi(t,x;\theta)$ and  $m_{\eta} := \chi(t,x;\eta)$ where $\theta \in \mathbb{R}^K$ and $\eta \in \mathbb{R}^K$ are the respective neural networks parameters to be learned. We denote for $x = (x_1, x_2, \dots, x_d) \in \mathbb{R}^d$ and $ z \in \mathbb{R}$, $(x^{-k}, z) := (x_1, \dots, x_{k-1}, z, x_{k+1}, \dots, x_d)$. We then consider the following loss function : 
\begin{align*}
    \mathcal{L}(u, m) = & \, \, C^{(HJB)} \mathcal{L}^{(HJB)}(u, m) +  C^{(KFP)}  \mathcal{L}^{(KFP)} (u, m) + C^{(init)} \mathcal{L}^{(init)}(u, m) \\ 
    & +  C^{term}  \mathcal{L}^{(term)} (u, m) + C^{(norm)} \mathcal{L}^{(norm)}(u, m) + C^{(period)} 
 \mathcal{L}^{(period)}(u, m),
\end{align*}
where, using the notation $F[m]=  F(x , m(t,x) )$,  
\begin{align*}
    & \mathcal{L}^{(HJB)}(u, m)  = \Big\lVert -\partial_t u - \frac{1}{2} \Delta u + \frac{1}{2} \lvert \nabla u \rvert^2 - F[m] \Big\rVert_{L^2([0,T] \times \mathbb{T}^d)}^2, \\
    & \mathcal{L}^{(KFP)}(u, m)  = \Big\lVert \partial_t m  - \frac{1}{2} \Delta m - \operatorname{div} (m \nabla u) \Big\rVert_{L^2([0,T] \times\mathbb{T}^d)}^2, \\
    & \mathcal{L}^{(init)}(u, m) = \Big \lVert m(0, \cdot) - m_0 \Big \rVert_{L^2(\mathbb{T}^d)}^2, \\ 
    & \mathcal{L}^{(term)} (u, m) = \Big \lVert u(T, \cdot) - G( \cdot, m_T) \Big \rVert_{L^2(\mathbb{T}^d)}^2,  \\ 
    &  \mathcal{L}^{(norm)}(u, m) = \Big \lvert \int_{\mathbb{T}^d} u(t,x) dx \Big\rvert + \Big\lvert \int_{\mathbb{T}^d} m(t,x) dx - 1 \Big\rvert, \\
    & \mathcal{L}^{(period)}(u, m) =  \sum_{k=1}^d \lvert u(t,(x^{-k}, 0)) - u(t, (x^{-k}, 1) ) \rvert^2 + \sum_{k=1}^d  \lvert m(t,(x^{-k}, 0)) - m(t, (x^{-k}, 1) ) \rvert^2 . 
\end{align*}
and $ C^{(HJB)}$, $C^{(KFP)}$, $C^{(init)}$, $C^{term}$, $C^{(norm)}$, and $C^{(period)}$ are positive constants controlling the relative importance of each loss component. 

\vskip 1 \baselineskip

The objective is to find the right set of parameters $\theta$, $\eta$, so that the neural networks $\phi$ and $\chi$ minimize an approximation $L(u_{\theta}, m_{\eta})$ of the loss functional $\mathcal{L}(u , m)$ computed on a mini-batch: 
\begin{align}
   L(u_{\theta}, m_{\eta}) = & \, \, C^{(HJB)} L^{(HJB)}(u, m) +  C^{(KFP)}  L^{(KFP)} (u, m) + C^{(init)} L^{(init)}(u, m) \nonumber\\ 
    & +  C^{term} L^{(term)} (u, m) + C^{(norm)} L^{(norm)}(u, m) + C^{(period)} L^{(period)}(u, m),  \label{eq:disc_loss}
\end{align}
We use a Monte Carlo method to approximate the $L^2$ norms and compute the penalty terms as: 
\begin{align*}
    & L^{(HJB)}(u_{\theta}, m_{\eta})  = \frac{1}{M} \sum_{k=1}^{M}  \Big\lvert - \partial_t u_{\theta}(t_k, x_k) - \frac{1}{2} \Delta  u_{\theta}(t_k, x_k) + \frac{1}{2} \lvert \nabla  u_{\theta}(t_k, x_k) \rvert^2 - F[m_{\eta}](x_k) \Big\rvert^2, \\
    & L^{(KFP)}(u_{\theta}, m_{\eta})  = \frac{1}{M} \sum_{k=1}^{M}  \Big\lvert \partial_t m_{\eta}(t_k, x_k)- \frac{1}{2} \Delta m_{\eta}(t_k, x_k) - \operatorname{div} \Big(m(t_k, x_k) \nabla  u_{\theta}(t_k,x_k) \Big)  \Big\rvert^2, \\
    & L^{(init)}(u_{\theta}, m_{\eta})  = \frac{1}{M} \sum_{k=1}^{M}  \Big\lvert  m_{\eta}(0, x^{i}_k) - \rho_0(x^{i}_k) \Big \rvert^2,  \\ 
    & L^{(term)} (u_{\theta}, m_{\eta})  = \frac{1}{M} \sum_{k=1}^{M}  \Big \lvert u_{\theta}(T, x^{t}_k) - G(x^t_k, m_{\eta}(T,x^t_k) \Big \rvert^2,  \\ 
    & L^{(norm)}(u_{\theta}, m_{\eta})  = \Big \lvert \frac{1}{M} \sum_{k=1}^M u_{\theta}(t_k,x_k) \Big\rvert + \Big\lvert \frac{1}{M} \sum_{k=1}^M  m_{\eta}(t_k,x_k) - 1 \Big\rvert, \\
    & L^{(period)}(u_{\theta}, m_{\eta})  =  \frac{1}{M} \sum_{l=1}^M \sum_{k=1}^d \lvert u_{\theta}(t,(x_{l}^{-k}, 0)) - u_{\theta}(t, (x_{l}^{-k}, 1) ) \rvert^2  + \frac{1}{M} \sum_{l=1}^M \sum_{k=1}^d  \lvert m_{\eta}(t,(x_l^{-k}, 0)) - m_{\eta}(t, (x_l^{-k}, 1) ) \rvert^2 . 
\end{align*}

The pseudo-code is described in Algorithm \ref{algo:1}.  

%%%%%%%%%%%%%%%%%%%%%%%%%%%%%%%%%%%%%%%%%%%%%%%%%%%%%%%%%%%%%%%%%%%%%%%%%%%%%%%%%%%%%%%%%%%%%%%%%

\subsection{Accelerated Version of the Deep Galerkin Method for the Finite Horizon PDE System} 

In addition to the terms in the baseline DGM version,  we incorporate the turnpike estimates \eqref{eq:lin_tp_u} and \eqref{eq:lin_tp_m} into the loss function by considering : 
\begin{align*}
    \mathcal{L}^{acc}(u, m) = & \, \,  \mathcal{L}(u, m) 
 +  C^{(TPK)}_u  \mathcal{L}^{(TPK)}_u (u, m) + C^{(TPK)} _m \mathcal{L}^{(TPK)}_m (u, m), 
\end{align*}
where 
\begin{align}
      & \mathcal{L}^{(TPK)}_u(u, m)  = \int_{\delta T}^{(1-\delta)T}  \Big\lVert u(t,\cdot) - \langle u(t,\cdot) \rangle - \overline{u} \Big\rVert_{L^1(\mathbb{T}^d)}  \cdot (T-t) \cdot  \Big(e^{-\omega t} + e^{-\omega (T-t)}\Big)^{-1}  dt , \\
    & \mathcal{L}^{(TPK)}_m(u, m)  = \int_{\delta T}^{(1-\delta)T}  \Big\lVert m(t,\cdot) - \overline{m} \Big\rVert_{L^1(\mathbb{T}^d)} \cdot t \cdot \Big(e^{-\omega t} + e^{-\omega (T-t)}\Big)^{-1}  dt. 
\end{align}

We apply the DGM described above with the new loss functional $$L^{acc}(u_{\theta},m_{\eta}) = L(u_{\theta},m_{\eta}) + C^{(TPK)}_u L^{(TPK)}_u(u_{\theta},m_{\eta}) +  C^{(TPK)}_m L^{(TPK)}_m(u_{\theta},m_{\eta})$$
where, by denoting $M_t$ the number of sampled times and $M_x$ the number of sampled space points, 
\begin{align}
      & L^{(TPK)}_u(u_{\theta}, m_{\eta})  = \frac{1}{M_t} \sum_{k=1 : t_k \in [\delta T, (1-\delta)T]}^{M_t}  \frac{T-t_k}{M_x} \sum_{l=1}^{M_x} \Big\lvert u_{\theta}(t_k,x_l) - \langle u_{\theta}(t_k,x_l) \rangle - \overline{u}(x_l) \Big\rvert \cdot \Big(e^{-\omega t_k} + e^{-\omega (T-t_k)}\Big)^{-1} , \\  
    & L^{(TPK)}_m(u_{\theta}, m_{\eta})  = \frac{1}{M_t} \sum_{k=1: t_k \in [\delta T, (1-\delta)T]}^{M_t}  \frac{t_k}{M_x} \sum_{l=1}^{M_x} \Big\lvert m_{\eta}(t_k,x_l) - \overline{m}(x_l) \Big\rvert \cdot \Big(e^{-\omega t_k} + e^{-\omega (T-t_k)}\Big)^{-1} . 
\end{align}

\begin{algorithm}[ht]
    \begin{algorithmic}
        \Require Initialize the parameters of the neural networks and the learning rate $\alpha$. Maximum number of iterations $N$ and mini-batches of size $M$. 
        \For{n = 1, \dots, N} 
        \State Generate a training sample from the time interval and the domain's interior, say $\Omega$, initial and final conditions, i.e. : 
        \begin{itemize}
            \item Sample $\{(t_k, x_k)\}_{k=1}^M$ from $[0,T] \times \Omega$ according to certain distribution,
            \item Sample $\{x^{i}_k\}_{k=1}^M$ from $\Omega$ according to uniform distribution,
            \item Sample $\{x^{t}_k\}_{k=1}^M$ from $\Omega$ according to uniform distribution. 
        \end{itemize}
        \State Compute the loss functional $L(u_{\theta}, m_{\eta})$ from \eqref{eq:disc_loss} for the mini-batch $s_n = \{(t_k, x_k), x^{i}_k, x^{t}_k\}_{k=1}^M$. 
        \State Similarly, generate a validation sample and compute the corresponding loss functional. 
        \State Take a gradient descent step for $\theta$ and $\eta$ at the point $s_n$:  
        $$ \theta_{n+1} = \theta_n - \alpha \nabla_{\theta}L(u_{\theta}, m_{\eta}),$$
        $$ \eta_{n+1} = \eta_n - \alpha \nabla_{\eta}L(u_{\theta}, m_{\eta}).$$
        \EndFor
        \State 
    \end{algorithmic}
    \caption{Pseudo-code of the DGM algorithm for a finite-horizon MFG}
    \label{algo:1}
\end{algorithm}

%%%%%%%%%%%%%%%%%%%%%%%%%%%
\vskip 1 \baselineskip

\subsection{Implementation} 

As in \cite{SIRIGNANO20181339}, we use the Adam optimizer for the stochastic gradient descent. The estimation of the first-order and second-order moments of the gradients enables a dynamical adjustment of the learning rate for each parameter. The global learning rate schedule is chosen as a linear decay from initial rate of $10^{-2}$ to $10^{-5}$ in 300,000 iterations. The exponential decay rates of the moment estimation are $\beta_1 = 0.9$, $\beta_2 = 0.999$ and the numerical stability constant is $\epsilon = 10^{-7}$. The neural networks are initialized with the Xavier method so that the variance of the activations are identical across all layers, thus preventing the gradients from exploding or vanishing. The hyperparameters for the cost function take the values described in Table \ref{tab:hyperparams}.  

\begin{table}[ht]
    \centering
    \begin{tabular}{|c|c|c|c|c|c|c|c|c|c|c|}
    \hline 
        Algorithm & $C^{(HJB)}$ & $C^{(KFP)}$ & $C^{(init)}$ & $C^{(term)}$ & $C^{(norm)}$ & $C^{(period)}$ & $C^{(TPK)}_u$ & $C^{(TPK)}_m$ & $\delta$ \\
        \hline 
        DGM & 50 & 1 & 100 & 600  & 50 & 25 & 0 & 0 & 0 \\
        \hline 
        DGM-TP & 50 & 1 & 100 & 600  & 50 & 25 & 10 & 100 & 0.1 \\
        \hline 
    \end{tabular}
    \caption{Values of the hyperparameters}
    \label{tab:hyperparams}
\end{table}

The difficulty of tuning the hyperparameters has been underlined in \cite{carmona_convergence_2021}. Giving too much importance to one of the loss terms can lead to sub-optimal solutions that only minimize one penalization term while disregarding the other terms. A noteworthy point is the fact that the hyperparameters for the penalty terms common to both versions have been fine-tuned based on the performance for the baseline algorithm. To enable a relatively fair comparison we only fine-tune the turnpike penalty costs $C^{(TPK)}_u$ and $C^{(TPK)}_m$ for the turnpike-accelerated method. As for the neural networks, we use an architecture with $2$ hidden layers and a width equal to $100$. Despite the small number of hidden layers, we continue to use the terminology \emph{deep} to conform to what is used in the existing literature on the subject. The activation functions are chosen to be \textit{sigmoid}, except for the output layers which are set to be the \textit{identity} for the value function neural network and \textit{exponential} for the density neural network. The training relies on the sampling of new mini-batches at each SGD iteration with 10,240 points in time and space (in particular $M_t = 10$ for time and $M_x = 1,024$ for space) for $t \in (0,T)$ and 1,024 points for the initial and terminal conditions. The mini-batches are generated by : 
\begin{itemize}
    \item Sampling the time points on $[0,T]$ according to a scaled Beta distribution $\operatorname{Beta}(a,b)$ with parameters $a=b=0.5$. This has been found to be empirically better than a uniform sampling distribution : intuitively, this comes from the forward-backward nature of the PDE system and the fact that the initial and final conditions need to be learned for the neural networks to not learn a trivial solution minimizing only the PDE residuals while disregarding the initial and terminal condition penalties. 
    \item Sampling the space points uniformly on $[0,1]$ for the initial time $t=0$, the terminal time $t = T$ and the intermediate times in $(0,T)$. 
\end{itemize}

The convergence of the DGM method is ensured by monitoring the convergence of each validation penalty term towards zero. 

\begin{figure}[hbt!]
    \centering
    \begin{subfigure}{0.45\textwidth}
        \includegraphics[width = \textwidth]{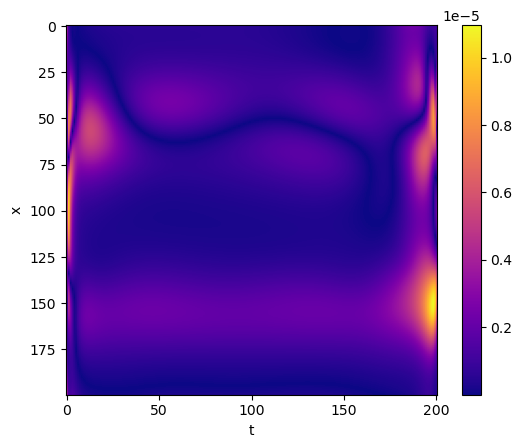}
    \end{subfigure}
    ~ 
    \begin{subfigure}{0.45\textwidth}
        \includegraphics[width = 0.975\textwidth]{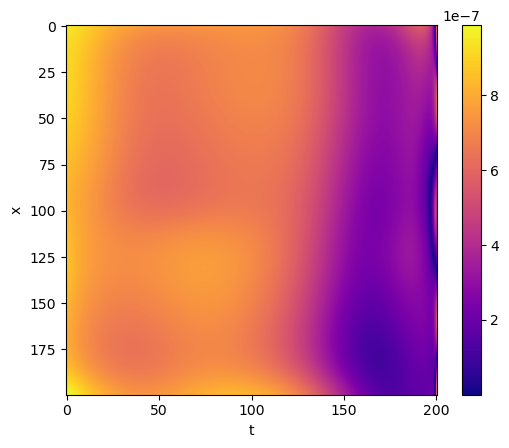}
    \end{subfigure}
    \caption{ Relative absolute error for the density (on the left) and the value function (on the right). The axes correspond to the index on a uniform grid of size 200 for time in $[0,T]$ and space in $[0,1]$.}
    \label{fig:heatmap_linear}
\end{figure}

%%%%%%%%%%%%%%%%%%%%%%%%%%%%%%%%%%%%%%%%%%%
\subsection{Numerical Results and Analysis}

This subsection presents the solutions obtained for the different methods considered so far : the Finite Difference Method (FDM), the baseline Deep Galerkin Method (DGM) and the turnpike-accelerated DGM method (DGM-TP).

\vskip 1 \baselineskip

To illustrate the performance of the DGM-TP version, we give in Figure \ref{fig:heatmap_linear} heat maps of the relative error of the DGM-TP solution with respect to the finite difference solution (considered as the exact solution) evaluated on a uniform $200 \times 200$ grid of the domain $[0,T]\times Q$.  For the density (left pane of Figure \ref{fig:heatmap_linear}), the errors are concentrated on the transient phases of the turnpike phenomenon, namely the initial and terminal intervals $[0, \delta T]$ and $[(1-\delta)T,T]$, while there is little error in the long middle phase, which corresponds to being on the turnpike. For the value function (right pane of Figure \ref{fig:heatmap_linear}), the size of the relative error is orders of magnitude smaller, and the best approximation is achieved at the terminal time, which is quite natural because of the enforcement of the terminal condition through the corresponding penalty $L^{(term)}$. The approximation error is higher away from the terminal time.

\begin{figure}[hbt!]
    \centering
    \begin{subfigure}{0.45\textwidth}
        \includegraphics[width = \textwidth]{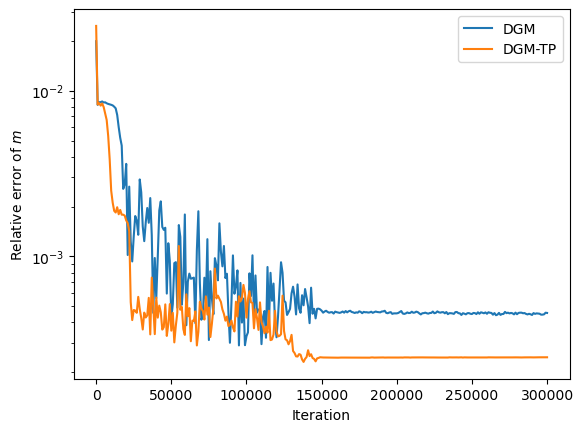}
    \end{subfigure}
    ~ 
    \begin{subfigure}{0.45\textwidth}
        \includegraphics[width = \textwidth]{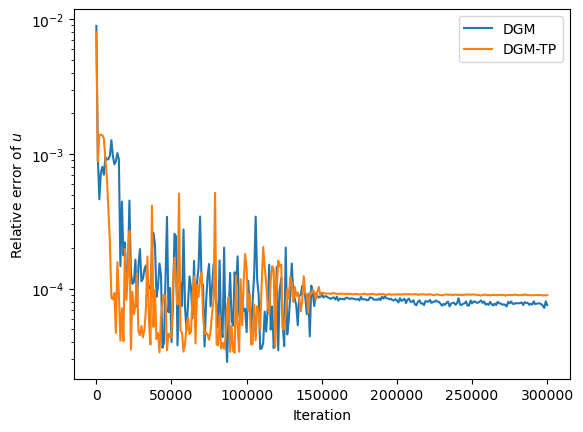}
    \end{subfigure}
    \caption{Relative absolute error for the density (left) and the value function (right) with respect to the training iteration index.}
    \label{fig:SGD_linear}
\end{figure}

\vskip 1 \baselineskip

We compare the performance between the DGM and DGM-TP neural network predictions on the same $200 \times 200$ grid with respect to the SGD iterations in Figure \ref{fig:SGD_linear}. The turnpike-accelerated method brings improvement to the approximation of the density, while there is no significant gain for the value function. This can be explained by the fact that the turnpike estimates involve the exact density function, while the one for the value function involves a mean-adjusted version, so the first estimate brings in more information about the system than the simple PDE equations. In particular, while monitoring the training of the value function, we notice that without any prior turnpike information, the baseline DGM neural network actually learns the shape of the correct value function quite well (after a relatively low number of iterations). As a result, adding the turnpike penalty to the optimization does not bring significant improvement for the approximation of the value function. 

\vskip 1 \baselineskip

\begin{figure}[hbt!]
    \centering 
\begin{subfigure}{0.3\textwidth}
  \includegraphics[width=\linewidth]{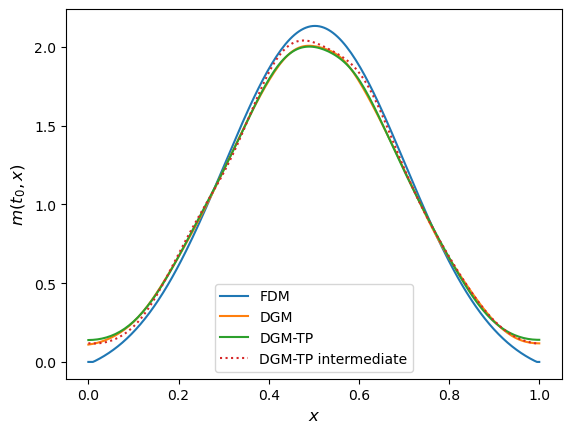}
\end{subfigure}\hfil 
\begin{subfigure}{0.3\textwidth}
  \includegraphics[width=\linewidth]{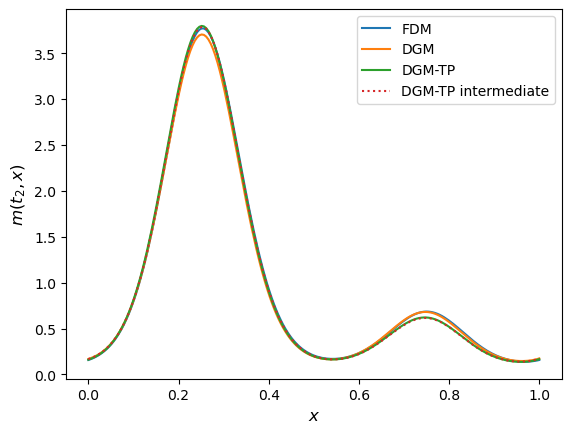}
\end{subfigure}\hfil 
\begin{subfigure}{0.3\textwidth}
  \includegraphics[width=\linewidth]{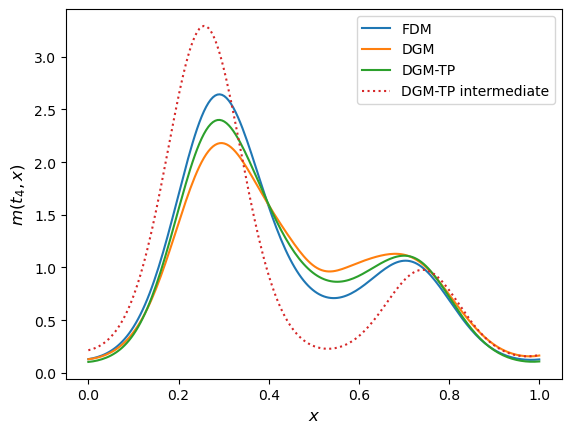}
\end{subfigure}

\medskip
\begin{subfigure}{0.3\textwidth}
  \includegraphics[width=\linewidth]{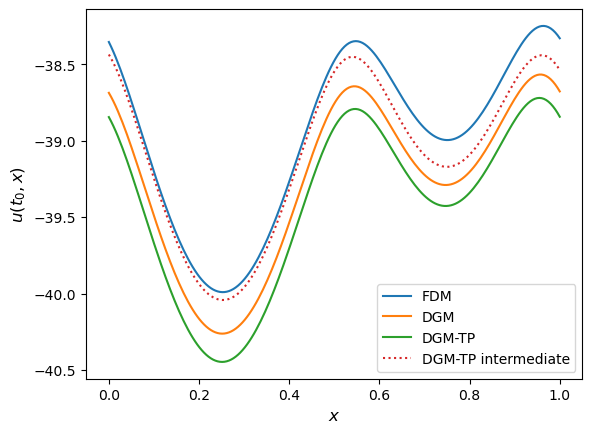}
\end{subfigure}\hfil 
\begin{subfigure}{0.3\textwidth}
  \includegraphics[width=\linewidth]{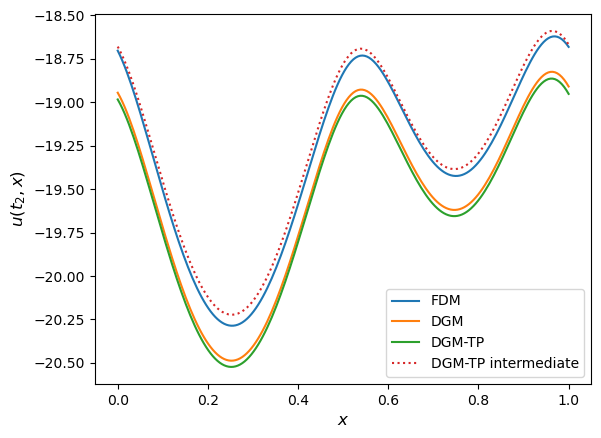}
\end{subfigure}\hfil 
\begin{subfigure}{0.3\textwidth}
  \includegraphics[width=\linewidth]{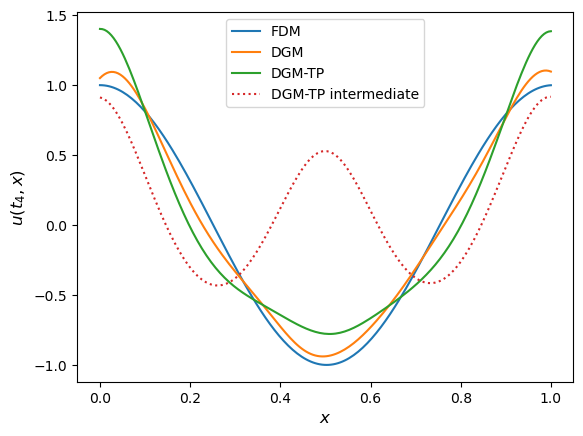}
\end{subfigure}
\caption{Plots of the density (top) and the value function (bottom) at three different times : $t_0 = 0$, $t_2 = T/2$, $t_4 = T$. The solutions obtained by Finite Differences appear in blue while the orange and green plots correspond to the iteration $i = 300,000$ of the DGM and DGM-TP methods respectively. The red-dotted plot corresponds to the DGM-TP solution at an intermediate iteration $i = 50,000$ (before full convergence of the algorithm). } 
\label{fig:main_linear_plots}
\end{figure}

Figure \ref{fig:main_linear_plots} displays the graphs of the numerical solutions at different times $t \in \{0, \frac{T}{2}, T\}$. Both DGM and DGM-TP solutions are consistent with the solution provided by the Finite Difference method. The density plots show that the DGM-TP density is closer to the FDM one, especially near the terminal time. Both DGM and DGM-TP value functions are close to each other, though slightly away from the FDM value function. The slight difference between them seems to stem from the difference in approximation performance at the terminal time and a progressive accumulation as the time moves to the initial time. The terminal approximation error may be caused by the introduction of additional turnpike terms. 
We produce an intermediate plot of the solution at an early training stage (iteration 50,000) of the DGM-TP algorithm
to show that the convergence of the iterative optimization is fast. We indeed have a very precise idea of the density and the value function, except at the terminal time $T$. At the terminal time, the red-dotted line is very far from the terminal condition in blue. It will take time to learn this information and for the red-dotted solution to slowly morph into the final DGM-TP solution in green, at the cost of less precision on the rest of the time interval. 

%%%%%%%%%%%%%%%%%%%%%%%%%%%%%%%%%%%%%%%%%%%%%%%%%%%%%%%%%%%%%%%%%%%%%%%%%%%%%%
\section{\textbf{A New Class of Linear-Quadratic MFG Models with the Turnpike Property}}
 \label{section:4}
Rather than aiming at the greatest possible generality, we introduce a  class of simple models which can be solved explicitly, allowing us to identify the turnpike property readily.

%%%%%%%%%%%%%%%%%%%%%%%%%%%%%%%%%%%%%%%%%%%%%%%%%%%%%%%%%%%%%%%% 
\subsection{State Dynamics and Instantaneous Costs} 
For the sake of simplicity, we assume that the generic agent controls the dynamics of their own state living in $\mathbb{R}$: 
\begin{equation}
    \begin{cases}
    & dX_t = -\alpha_t dt + \sigma dW_t, \\
    & X_0 = \xi \sim m_0 ,
    \end{cases}
\end{equation}
where the control process $\balpha=(\alpha_t)_{t\ge 0}$ takes values in a closed convex subset $A \subset \mathbb{R}$, $\bW=(W_t)_{t\ge 0}$ is a given one-dimensional Brownian motion, and $m_0\in\cP(\RR)$ is the initial distribution of the population of agents. 
As in Section \ref{section:2}, we assume that the running cost has a separablee structure, i.e.: 
$$
L(x,\alpha, m) = \frac{1}{2} \lvert \alpha \rvert^2 + F[m](x),
$$ 
where $L: \mathbb{R} \times A \times  \mathcal{P}(\mathbb{R}) \mapsto \mathbb{R} $ is assumed to be continuous, and where the operator $F[m]$ is here associated to a non-local coupling:
$$
F[m](x) = f(x,m) = qx^2  + \beta x \cdot \int ydm(y) + \gamma \Big( \int y dm(y) \Big)^2.
$$

%%%%%%%%%%%%%%%%%%%%%%%%%%%%%%%%%%%%%%%%%%%%%%%%%%%%%%%%%%%%%%%% 
\subsection{Explicit Solution of the Ergodic MFG Problem} 
In the ergodic version of the problem, the generic agent wants to minimize the following ergodic cost: 

\begin{equation}
    \mathcal{J} (x,\alpha) = \underset{T \rightarrow + \infty}{\lim \inf} \, \frac{1}{T}
 \mathbb{E} \Big[ \int_0^T L(X_t, \alpha_t, m_t)dt \Big],
\end{equation}
where the flow $\bm=(m_t)_{t\ge 0}$ of probability measure is stationary in the sense that for every $t\ge 0$, $m_t=m$ for some $m \in \mathcal{P}(\mathbb{R})$, and the control process $\balpha$ is also stationary, namely given by a fedback function independent of time, i.e. $\alpha_t=\alpha(X_t)$ for all $t\ge 0$.
Using the notation $\nu := \sigma^2/2$, the value function $v(x)=\inf_{\balpha} \mathcal{J}(x,\balpha)$ and the stationary distribution $m$ are found by solving the ergodic PDE system becomes: 
 \begin{subequations} \label{eq:eLQMFG}
    \begin{empheq}[left={(LQMFG-e):= \empheqlbrace}]{align}
& - \nu \Delta v + \frac{1}{2} \vert \nabla v \rvert^2 - \lambda =  F[m](x) , \\
& \nu \Delta m + \operatorname{div}(\nabla v \cdot m) = 0,  \\
& \min v(x)  = 0, \quad  \int_{\mathbb{T}} m(x)dx = 1, \, m>0. 
    \end{empheq}
    \end{subequations}

The special form of our model was chosen to appeal to an existing solution of the problem. The following result was proven by Bardi as Theorem 4.1 in \cite{Bardi_explicit_sol}. In the current setting, it takes the simple form : 
\begin{theorem}
    Assuming that $\overline{q}>0$ and $2 \overline{q} \neq -\beta $. Then, 
    \begin{itemize}
        \item the system \eqref{eq:eLQMFG} has exactly one solution $(\lambda, v,m)$ of the quadratic-Gaussian form : 
        \begin{subequations}
    \begin{empheq}[left=\empheqlbrace]{align}
&  v(x) = \frac{(x-\mu)^2}{2s},  \quad m(x) = \frac{1}{\sqrt{ \pi s \sigma^2 }} \exp \Big( - \frac{(x-\mu)^2 }{s \sigma^2} \Big), \quad s = \frac{1}{\sqrt{2q}}, \quad \mu = 0,\\ 
& \overline{\alpha}(x) = \frac{x - \mu}{s}, \quad  J(x, \overline{\alpha} ) = \lambda, \quad  \lambda = \frac{\nu}{s}- \frac{\mu^2}{2s^2} + \gamma\mu^2,
\end{empheq}
\end{subequations}

\item if in addition $\beta \geq 0$, then it is the unique solution of \eqref{eq:eLQMFG}. 
    \end{itemize}
\end{theorem} 

%%%%%%%%%%%%%%%%%%%%%%%%%%%%%%%%%%%%%%%%%%%%%%%%
\subsection{Finite Horizon Linear Quadratic MFG} 
In the finite horizon form of the model, for a finite horizon $T>0$, we minimize the expected cost:
\begin{equation}
    J (x,\alpha) = 
 \mathbb{E} \Big[ \int_0^T L(X_t, \alpha_t, m_t)dt + g(T,x_T,m_T)\Big],
\end{equation}
where the flow $\bm=(m_t)_{0\le t\le T}$ is now time dependent, and expected to coincide  in equilibrium with the distribution of the marginal distribution of the state, i.e. $m_t=\mathcal{L}(X_t)$. The terminal cost function $g$  is part of the data of the model. 

For the sake of definiteness, we choose a Gaussian initial distribution $m_0 = \mathcal{N}( \mu_0, \sigma_0)$, and a terminal cost function independent of the distribution, and penalizing the distance of the terminal state from a given value $r$, to be specific, $g(T,x_T,m_T) = \Psi (x-r)^2$. 

\vskip 2pt
The strategy to solve such a model is well known. See for example \cite[Section 3.5]{CarmonaDelarue_book_I}. We recall the main steps to derive the turnpike property.
The Hamiltonian takes the form: 
\begin{equation*}
H(x,m,p) = \underset{a}{\max} -  f(x,a,m) - b(x,a,m) \cdot p,
\end{equation*}
where, for deterministic constants $B$ and $Q$ to be determined below, 
\begin{equation*}
f(x,a,m) = \frac{1}{2}\Big( Q x^2 + B \Big(x- \int \xi m(\xi) d\xi\Big)^2 + a^2\Big), \quad \text{and}\quad b(x,a,m) = - a.
\end{equation*}
Hence, the first order optimality conditions provides: $a^* = p$, so that:
\begin{align*}
    & H(x,m,p) = \frac{1}{2} \Big(Q x^2 + B (x- \int \xi m(\xi) d\xi)^2\Big)   + \frac{1}{2} p^2,. 
\end{align*}
and $H_p(x,m,p) = p$, so that the HJB equation can therefore be written as: 
\begin{align*}
    & - \partial_t u(t,x) - \frac{\sigma^2}{2} \Delta u(t,x) + \frac{1}{2} \lvert \nabla u(t,x) \rvert^2 = \frac{1}{2} (Q x^2 + B(x-\int \xi m(\xi) d\xi)^2), \\
    & u(T,x)  = \Psi (x-r)^2. 
\end{align*}

We can then use the following ansatz $u(t,x) = \frac{1}{2} \phi_t x^2 + \chi_t x + \psi_t$ with  functions $\phi$, $\chi$,  and $\psi$ from $[0,T]$ to $\mathbb{R}$ to be determined. As the HJB equation only depends on the mean field through its mean $\mu_t=\int \xi m_t(\xi) d\xi$, we can in fact simplify the problem by reducing the KFP equation to an ODE for this mean. We get: 
$$
    \frac{d \mu}{dt} - \int m(t,x) H_p(x, m(t), \nabla u(t,x))) dx = 0,
$$
which is easily seen to be equivalent to 
$$
\frac{d \mu}{dt}  - (\phi_t  \mu_t + \chi_t)  = 0,
$$
by using the normalization condition. The PDE system thus simplifies to the following ODE system : 

 \begin{subequations}
    \begin{empheq}[left=\empheqlbrace]{align}
&  \frac{d \mu}{dt}  - (\phi_t  \mu_t + \chi_t)  = 0, & \mu_0 = m_0 , \\
& - \frac{d \phi}{dt} = - \phi_t^2 + Q + B, & \phi_T = 2 \Psi, \label{eq:riccati} \\ 
& - \frac{d \chi}{dt} = - \phi_t \chi_t - B \mu_t , & \chi_T = 2 \Psi r, \\
& - \frac{d \psi}{dt} = \frac{\sigma^2}{2} p_t -\frac{1}{2} \chi_t^2 + \frac{1}{2} B \mu^2_t , & \psi_T = \Psi r^2.  
\end{empheq}
\end{subequations}

%%%%%%%%%%%%%%%%%%%%%%%%%%%%%%%%%%%%%%%%%%%%%%%%%
\subsection{The Turnpike Estimates}

Equation \eqref{eq:riccati} is a Riccati equation, whose solution is given (after simplification) by: 
\begin{equation}
\phi_t = \sqrt{C} + \frac{2 \sqrt{C}(\sqrt{C} - \gamma)}{\sqrt{C} (e^{2 \sqrt{C}(T-t)} +1 ) + \gamma (e^{2 \sqrt{C}(T-t)} - 1 )} ,
\end{equation}
where $C = Q + B$, $\gamma = 2 \Psi$. 
We therefore have : 
\begin{align*}
    \lvert \phi_t - \sqrt{C} \rvert  = \frac{2 \sqrt{C}\lvert \sqrt{C}-\gamma\rvert}{\lvert \sqrt{C} (e^{2 \sqrt{C}(T-t)} +1 ) + \gamma (e^{2 \sqrt{C}(T-t)} - 1 ) \rvert} .
\end{align*}
Since  $2 \sqrt{C} (T-t) \geq 0$ for $t \in [0,T]$, we obtain that the denominator which we denote by $D$ is positive, and that $\lvert D \rvert \geq \sqrt{C} e^{2 \sqrt{C}(T-t)}$. Hence: 
\begin{equation}
    \lvert\phi_t - \sqrt{C} \rvert \leq 2 \lvert \sqrt{C}-\gamma\rvert e^{-2 \sqrt{C}(T-t)}, \quad t \in [0,T]
    \label{eq:LQ-TP-phi}.
\end{equation}
We rewrite the equations of $\mu$ and $\chi$ as a linear system: 

\begin{align}
    \dot{X}_t = A_t X_t, \quad  X_t = \begin{pmatrix}
        \mu_t \\ \chi_t
    \end{pmatrix}, \quad  A_t = \begin{pmatrix}
        - \phi_t & -1 \\ B & \phi_t
    \end{pmatrix}. 
\end{align}
The initial-terminal conditions make the system a forward-backward problem. Instead of solving the backward equation, we use a shooting method to find the right initial condition for $\chi_t$ to ensure that $\chi_T = 2\Psi r$. This is now a simple ODE system with a non-constant matrix $A_t$ that generally does not commute (indeed, for $\phi$ non identically constant and $t \neq s$, $A_t A_s \neq A_s A_t$), hence we do not have any closed-form expression for the solution. The case where $\sqrt{C} = \gamma$ is the only one for which we have such an explicit solution, and for the sake of deriving solutions from explicit formulas which we can compare to the results of our numerical algorithms, we shall restrict ourselves to this case from now on.
The ODE system becomes: 
\begin{align}
    \dot{X}_t = A X_t, \quad  A := \begin{pmatrix}
        - \sqrt{C} & -1 \\ B & \sqrt{C}
    \end{pmatrix}. 
\end{align}
Since $Q>0$, then $A$ has two distinct real eigenvalues $\lambda_1 = - \sqrt{C - B}$, $\lambda_2 = + \sqrt{C - B}$ with respective eigenvectors $v_1 = (- \frac{\sqrt{C} + \sqrt{C-B}}{B} ,1)^{\dagger}$ and $v_2 = (- \frac{\sqrt{C} -\sqrt{C-B}}{B} ,1)^{\dagger}$. The solution of the system can therefore be expressed as :
$ X_t = e^{Mt} X_0$ and the system has the phase diagram depicted in Figure \ref{fig:phase_portrait}.  

By a straightforward calculation, the correct initial condition to ensure $\chi_T = 2 \Psi r$ is given by : 
\begin{equation}
    c = \frac{2 \Psi r + (e^{\lambda_1 T } - e^{\lambda_2 T }) \frac{B}{2 \sqrt{C-B)}} \mu_0 }{0.5\cdot e^{\lambda_1 T} (1 - \frac{\sqrt{C}}{\sqrt{C-B)}} )- 0.5\cdot e^{\lambda_2 T } (1 +\frac{\sqrt{C}}{\sqrt{C-B)}} )}.
\end{equation}
We note that for $T$ large enough, $c \sim - \frac{B}{\sqrt{C-B} + \sqrt{C}} \mu_0$, so $c \cdot \mu_0 < 0$. Thus, we are interested in initial points located in the second and fourth quadrants of the phase diagram. We also notice that for $T$ large enough, the initial point $(\mu_0, c)$ will actually be in the vicinity of the linear function directed by $v_1$ and will then lead to a trajectory that ends up in the vicinity of the linear function directed by $v_2$. 
We can therefore infer that: 
\begin{equation}
    \lvert \mu_t \rvert + \lvert \chi_t \rvert = O \Big( e^{-\sqrt{C-B} t} + e^{-\sqrt{C-B} (T-t)} \Big), \quad \forall t \in [0,T]
    \label{eq:LQ-TP-mu-chi}
\end{equation}
where $ O(\cdot)$ is uniform with respect to $T$. 

\begin{figure}[hbt!]
    \centering
    \includegraphics[width = 7cm]{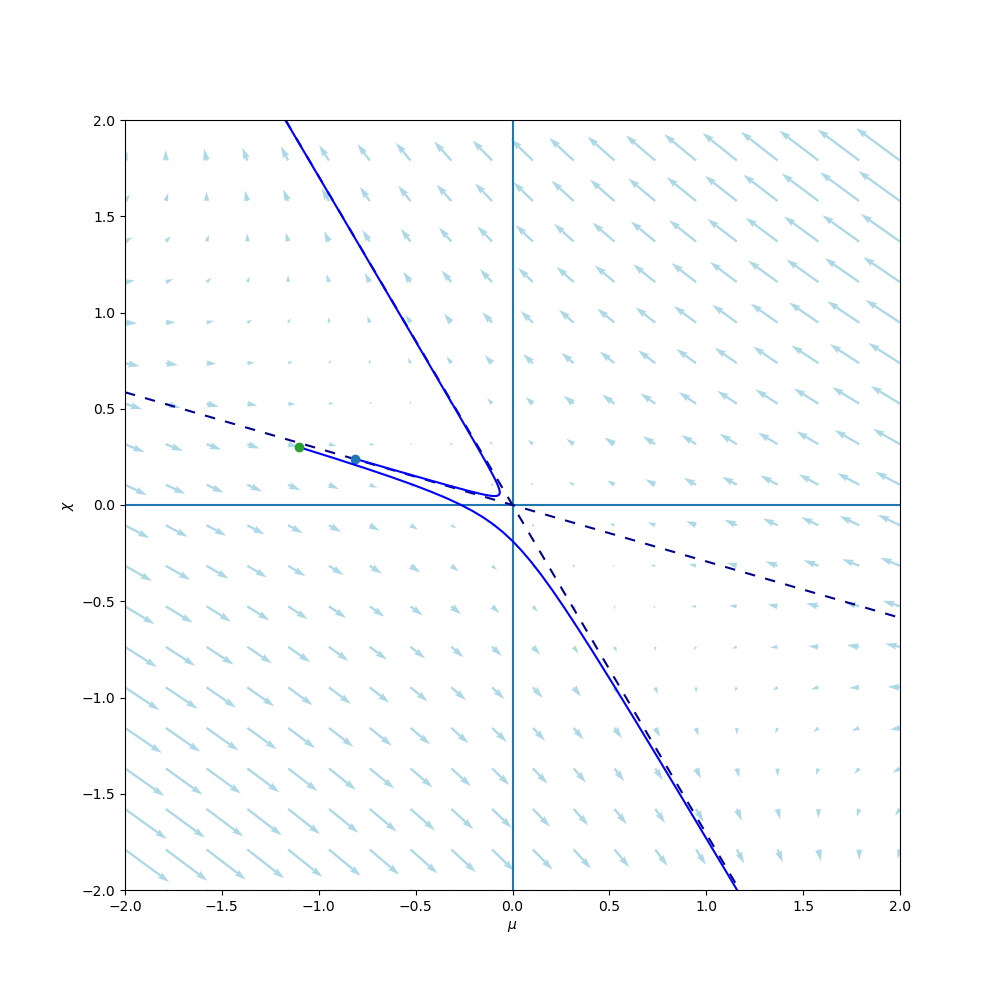}
    \caption{Phase plot of the ODE system for $(\mu, \chi)$ with two trajectories that start at the initial points marked by a circle. }
    \label{fig:phase_portrait}
\end{figure}

Matching the running cost coefficients of the running cost functions in the ergodic and finite horizon problems leads to $q = \frac{1}{2} (Q+B)$, $\beta = - B$ and $\gamma = \frac{1}{2} B$. Hence, the ergodic solution $\overline{u}$ satisfies $D\overline{u}(t,x) = \sqrt{C} x $. We have the following turnpike estimates:

\begin{proposition}
    Under the above assumptions,  for every $t\in [0,T]$,   
\begin{align}
 &  \int_x \lvert u(t,x) - u(t,0) - \overline u(x) \rvert \overline{m}(dx) \leq C \Big( e^{-\sqrt{C-B} t} + e^{-\sqrt{C-B} (T-t)}\Big),\\
   &  \int_x \lvert Du(t,x) - D\overline u(x) \rvert \overline{m}(dx) \leq C \Big( e^{-\sqrt{C-B} t} + e^{-\sqrt{C-B} (T-t)}\Big),\\
   & \lvert \mu_t \rvert  \leq C \Big( e^{-\sqrt{C-B} t} + e^{-\sqrt{C-B} (T-t)}\Big).
\end{align}

\end{proposition}

\begin{proof}
For the first inequality : 
\begin{align*}
    \int_x \lvert u(t,x) - u(t,0) - \overline u(x) \rvert \overline{m}(dx) & = \int_x \lvert \frac{1}{2} (\phi_t - \sqrt{C})x^2 + \chi_t x \rvert \overline{m}(dx) \\
   & \leq  \frac{1}{2} \lvert \phi_t - \sqrt{C} \rvert   \int_x x^2  \overline{m}(dx) + \lvert \chi_t \rvert \int_x x \overline{m}(dx) \\
    & \leq \frac{\sigma^2}{2} \cdot \lvert \phi_t - \sqrt{C} \rvert + \sqrt{\frac{2}{\pi}} \lvert \chi_t \rvert 
\end{align*}
by the triangular inequality. From \eqref{eq:LQ-TP-phi} and \eqref{eq:LQ-TP-mu-chi}, there exists a constant $K>0$ such that : 
\begin{align*}
    \int_x \lvert u(t,x) - u(t,0) - \overline u(x) \rvert \overline{m}(dx) & \leq K \cdot  \Big( e^{-\sqrt{C-B} t} + e^{-\sqrt{C-B} (T-t)}\Big) 
\end{align*}
For the second inequality : 
\begin{align*}
    \int_x \lvert Du(t,x) - D\overline u(x) \rvert \overline{m}(dx) & = \int_x \lvert (\phi_t - \sqrt{C})x + \chi_t \rvert \overline{m}(dx) \\
   & \leq  \lvert \phi_t - \sqrt{C} \rvert \int_x \lvert x  \rvert \overline{m}(dx) + \lvert \chi_t \rvert \int_x \overline{m}(dx) \\
    & \leq \sqrt{\frac{2}{\pi}} \cdot \lvert \phi_t - \sqrt{C} \rvert + \lvert \chi_t \rvert
\end{align*}
by the triangular inequality and the normalization condition. From \eqref{eq:LQ-TP-phi} and \eqref{eq:LQ-TP-mu-chi}, there exists a constant $K>0$ such that : 
\begin{align*}
    \int_x \lvert Du(t,x) - D\overline u(x) \rvert \overline{m}(dx) & \leq K \cdot  \Big( e^{-\sqrt{C-B} t} + e^{-\sqrt{C-B} (T-t)}\Big) 
\end{align*}
since $ 2 \sqrt{C} \geq \sqrt{C} \geq \sqrt{C-B}$. 

\end{proof}

Figure \ref{fig:turnpike_lqmfg} illustrates the turnpike property computed from the analytical solutions. 

\begin{figure}[hbt!]
    \centering
    \begin{subfigure}[b]{0.45\textwidth}
        \includegraphics[width = \textwidth]{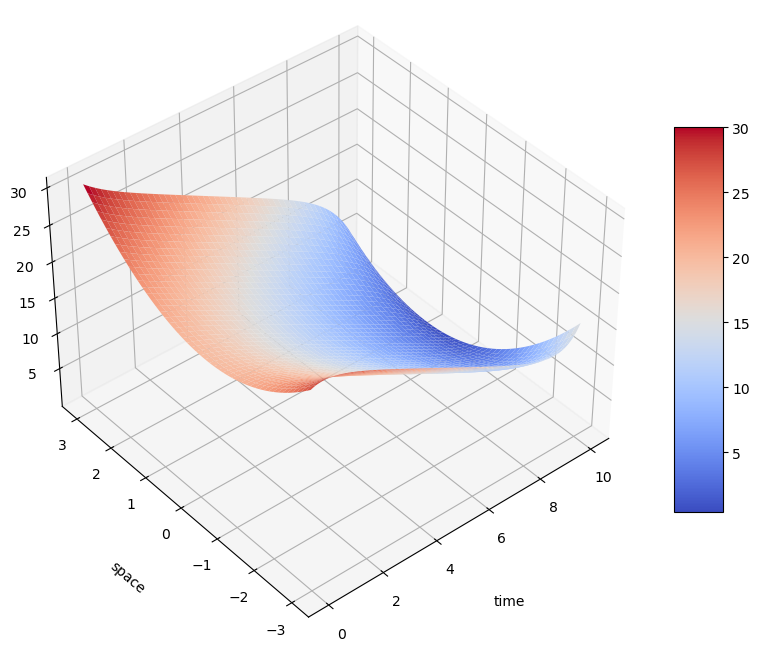}
        \subcaption{Value function $u(t,x)$}
    \end{subfigure}
    \hfill
    \begin{subfigure}[b]{0.45\textwidth}
        \centering
        \includegraphics[width = \textwidth]{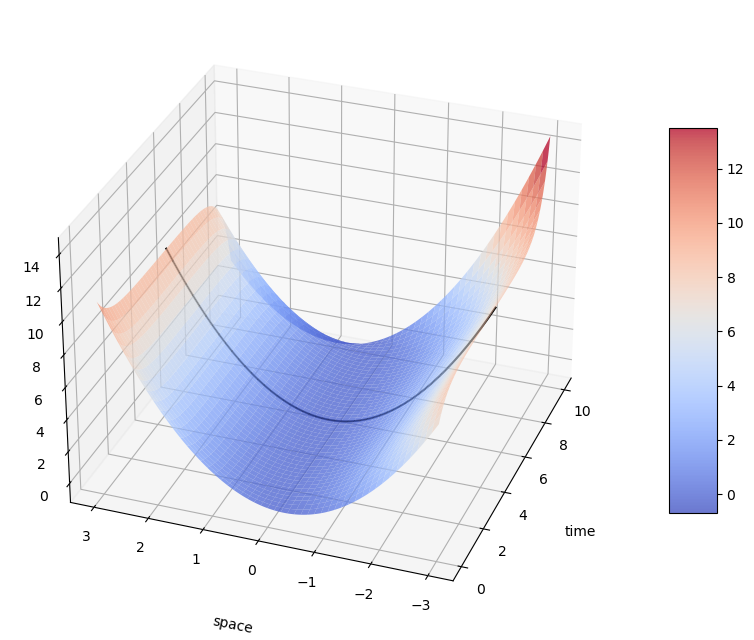}
        \subcaption{Centered value function $u(t,x) - u(t,0)$ and the ergodic value function $\overline{u}(x)$ plotted at $t = T/2$}
    \end{subfigure}
    \bigskip
    \begin{subfigure}[b]{0.45\textwidth}
        \includegraphics[width = \textwidth]{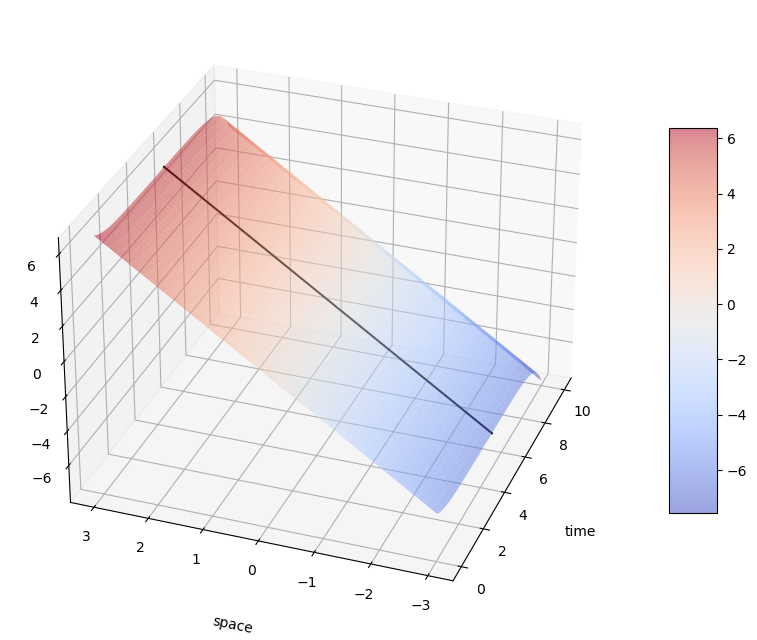}
         \subcaption{Space derivative of the value function $Du(t,x)$ and ergodic value function derivative $D \overline{u}(x)$ plotted at $t = T/2$}
    \end{subfigure}
    \hfill
    \begin{subfigure}[b]{0.45\textwidth}
        \centering
        \includegraphics[width = \textwidth]{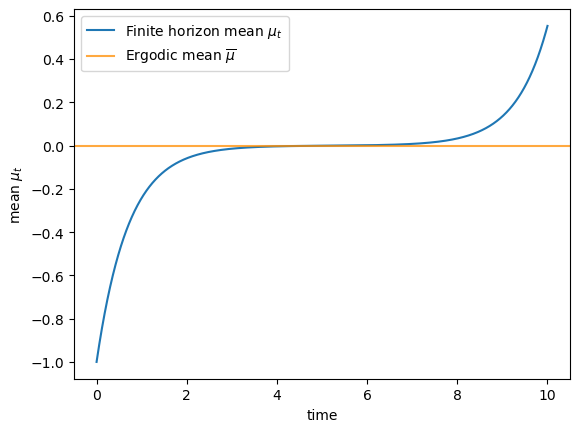}
         \subcaption{Average of the finite horizon $\mu_t$ and the ergodic average $\overline{\mu}$}
    \end{subfigure}
    \caption{Plots of the variation over time of the value function, its centered-at-the-origin version, its space derivative, and of the average $\mu(t,x) = \int x m_t(dx)$ together with the respective turnpikes.}
    \label{fig:turnpike_lqmfg}
\end{figure}

%%%%%%%%%%%%%%%%%%%%%%%%%%%%%%%%%%%%%%%%%%%%%%%%%%%%%%%%%%%%%
\subsection{Numerical Results}

We implement the baseline DGM method in the form presented in Section \ref{section:3}. We merely restrict the space on which we compute the solution to $D = [-L, L]$ for $L = 3$ and modify the Monte-Carlo estimates accordingly. 

\subsubsection{Turnpike accelerated DGM formulation}

Since we work on the bounded domain $D$, we no longer need the additional ergodic density factor in the turnpike estimates and consider a cost integrating them in the following fashion: 
$$\mathcal{L}^{acc}(u,m) = \mathcal{L}(u,m) + C_u \mathcal{L}^{(TPK-LQ)}_{u/Du}(u,m) +  C_m \mathcal{L}^{(TPK-LQ)}_m(u,m),$$
where
\begin{align}
        & \mathcal{L}^{(TPK-LQ)}_u(u, m)  = \int_{\delta T}^{(1-\delta)T}  \Big\lVert u(t,\cdot) - u(t,0) - \overline{u}\Big\rVert_{L^1(D)}  \cdot  \Big(e^{-\omega t} + e^{-\omega (T-t)}\Big)^{-1}  dt , \\
      & \mathcal{L}^{(TPK-LQ)}_{Du}(u, m)  = \int_{\delta T}^{(1-\delta)T}  \Big\lVert Du(t,\cdot) - D\overline{u}\Big\rVert_{L^1(D)}  \cdot  \Big(e^{-\omega t} + e^{-\omega (T-t)}\Big)^{-1}  dt , \\
    & \mathcal{L}^{(TPK-LQ)}_m(u, m)  = \int_{\delta T}^{(1-\delta)T} \Big\lvert \mu (t) - \overline{\mu} \Big\rvert
    \cdot \Big(e^{-\omega t} + e^{-\omega (T-t)}\Big)^{-1}  dt. 
\end{align}
where $\omega = \sqrt{C - B}$, where $\mu(t) = \int_{-L}^L x m(t, dx)$, and $\overline{\mu} = \int_{-L}^L x \overline{m}(dx)$. 

We apply the DGM described above with the new loss functional $$L^{acc}(u_{\theta},m_{\eta}) = L(u_{\theta},m_{\eta}) + C_u L^{(TPK-LQ)}_{u/Du}(u_{\theta},m_{\eta}) +  C_m L^{(TPK-LQ)}_m(u_{\theta},m_{\eta}),$$
where
\begin{align}
    & L^{(TPK-LQ)}_u(u_{\theta}, m_{\eta})  = \frac{1}{M_t} \sum_{k=1: t_k \in [\delta T, (1-\delta)T]}^{M_t}  \Big\{ \frac{2L}{M_x} \sum_{l=1}^{M_x} \Big\lvert u_{\theta}(t_k,x_l) -  u_{\theta}(t_k,0)- \overline{u}(x_l) \Big\rvert  \Big\} \cdot \Big(e^{-\omega t_k} + e^{-\omega (T-t_k)}\Big)^{-1} , \\
    & L^{(TPK-LQ)}_{Du}(u_{\theta}, m_{\eta})  = \frac{1}{M_t} \sum_{k=1: t_k \in [\delta T, (1-\delta)T]}^{M_t}  \Big\{ \frac{2L}{M_x} \sum_{l=1}^{M_x} \Big\lvert Du_{\theta}(t_k,x_l) - D\overline{u}(x_l) \Big\rvert \Big\} \cdot \Big(e^{-\omega t_k} + e^{-\omega (T-t_k)}\Big)^{-1} , \\
    & L^{(TPK-LQ)}_m(u_{\theta}, m_{\eta})  = \frac{1}{M_t} \sum_{k=1: t_k \in [\delta T, (1-\delta)T]}^{M_t} \Big\lvert  \frac{2L}{M_x} \sum_{l=1}^{M_x} x_l \cdot m_{\eta}(t_k,x_l) - \overline{\mu} \Big\rvert \cdot \Big(e^{-\omega t_k} + e^{-\omega (T-t_k)}\Big)^{-1} . 
\end{align}
where $\overline{\mu} = \frac{2L}{M_x} \sum_{l=1}^{M_x} x_l  \overline{m}(x_l)$. 

\subsubsection{Implementation} 

We use the exact analytical form of the ergodic quadratic-Gaussian solutions as knowledge of the turnpike, instead of having the intermediate step of solving the ergodic Mean Field Game as in Section \ref{section:3}. We therefore have $ \overline{u}(x) = \frac{1}{2}\sqrt{C}x^2 $ and $\overline{\mu} = 0$.  The implementation is similar to the one described in Section \ref{section:3} but with a learning rate schedule going from $10^{-2}$ to $10^{-6}$ in 400,000 SGD iterations.  The hyperparameters are given by Table \ref{tab:hyperparams_lqmfg}. The choice of $C^{(TPK)}_m = 0.1$ is motivated by the competing normalization condition on $m$ and a value of $1$ leads to a degradation of the normalization loss in the intermediate time interval on which we compute the turnpike loss. The solutions are then evaluated on a $2,000\times 2,000$ grid of $[0,T] \times [-L,L]$ for $T =10$ and $L=3$. The parameters of the linear-quadratic model are $Q = 2$, $B = 2$, $\Psi = 1$ and $r=1$. 

\begin{table}[ht]
    \centering
    \begin{tabular}{|c|c|c|c|c|c|c|c|c|c|c|}
    \hline 
        Algorithm & $C^{(HJB)}$ & $C^{(KFP)}$ & $C^{(init)}$ & $C^{(term)}$ & $C^{(norm)}$ & $C^{(TPK-LQ)}_{u}$ & $C^{(TPK-LQ)}_m$ & $\delta$\\
        \hline 
        DGM & 100 & 10 & 100 & 600  & 50 & 0 & 0 & 0 \\
        \hline 
        DGM-TP & 100 & 10 & 100 & 600 & 50 & 1 & 0.1 & 0.2 \\
        \hline 
    \end{tabular}
    \caption{Values of the hyperparameters for the DGM methods for the LQ MFG model.}
    \label{tab:hyperparams_lqmfg}
\end{table}

Given the different turnpike estimates derived in the previous subsection, we implement two turnpike-accelerated versions of the DGM algorithm with  on one hand, the turnpike estimates involving $u$ and $\mu$, and on the other hand the ones involving $Du$ and $\mu$. These are respectively denoted by DGM-TP $u$ and DGM-TP $Du$. 

\subsubsection{Results and comparisons}

We present in Figure \ref{fig:3D_value_function_comparison} the plots of the value functions obtained by the baseline DGM method and the turnpike-accelerated versions in comparison to the semi-explicit analytical solution. The left plot suggests that simply following the generic method of monitoring the convergence of the loss terms to 0 is not enough to provide a good approximation of the solution with the chosen set of hyperparameters. With the same values, the addition of the turnpike estimates helps the neural network approximating the value function to learn the correct shape and this results in a smaller relative error as depicted in Figure \ref{fig:rel_err_lqmfg}.

\begin{figure}[hbt!]
    \centering
    \begin{subfigure}{0.45\textwidth}
        \includegraphics[width = \textwidth]{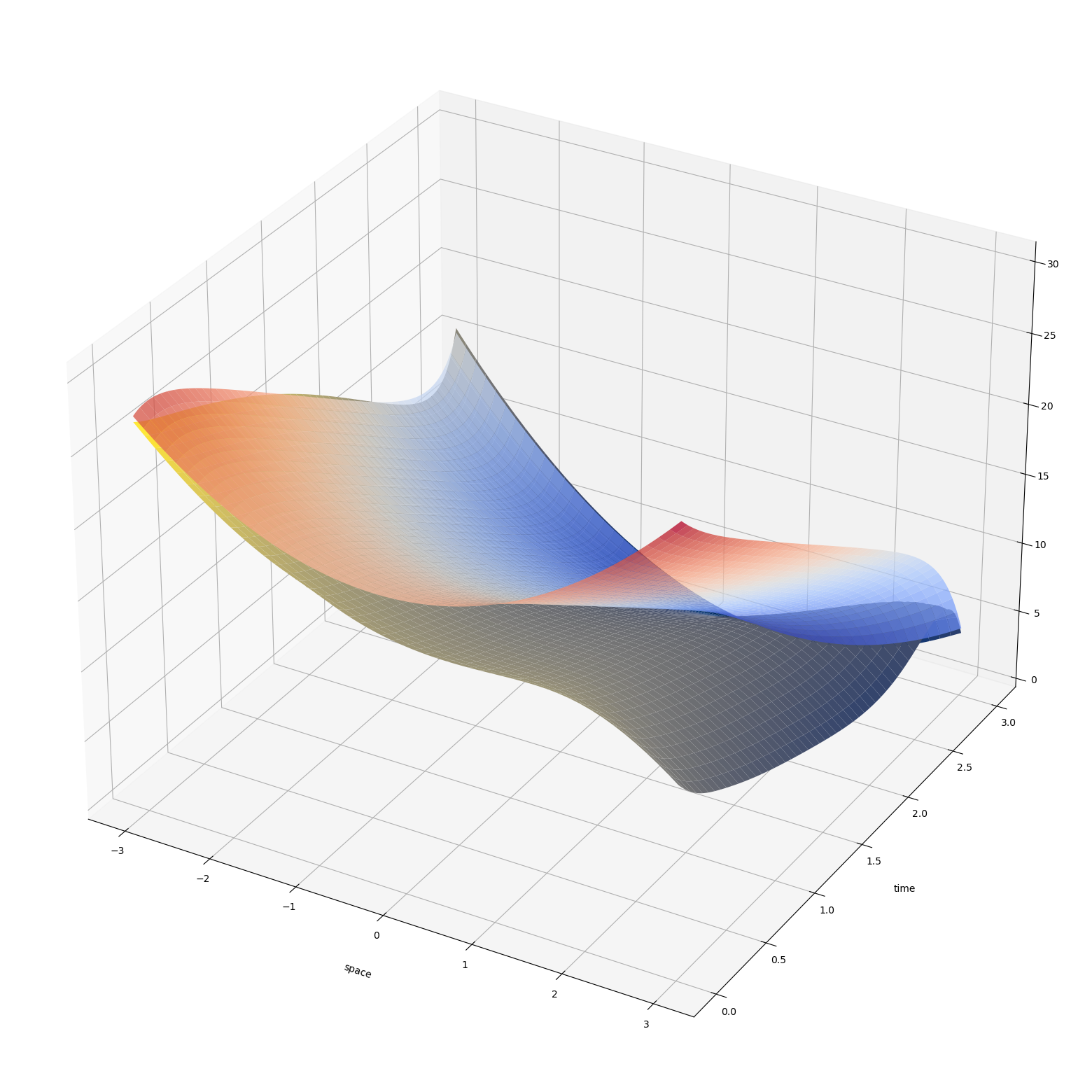}
    \end{subfigure}
    ~ 
    \begin{subfigure}{0.45\textwidth}
        \includegraphics[width = \textwidth]{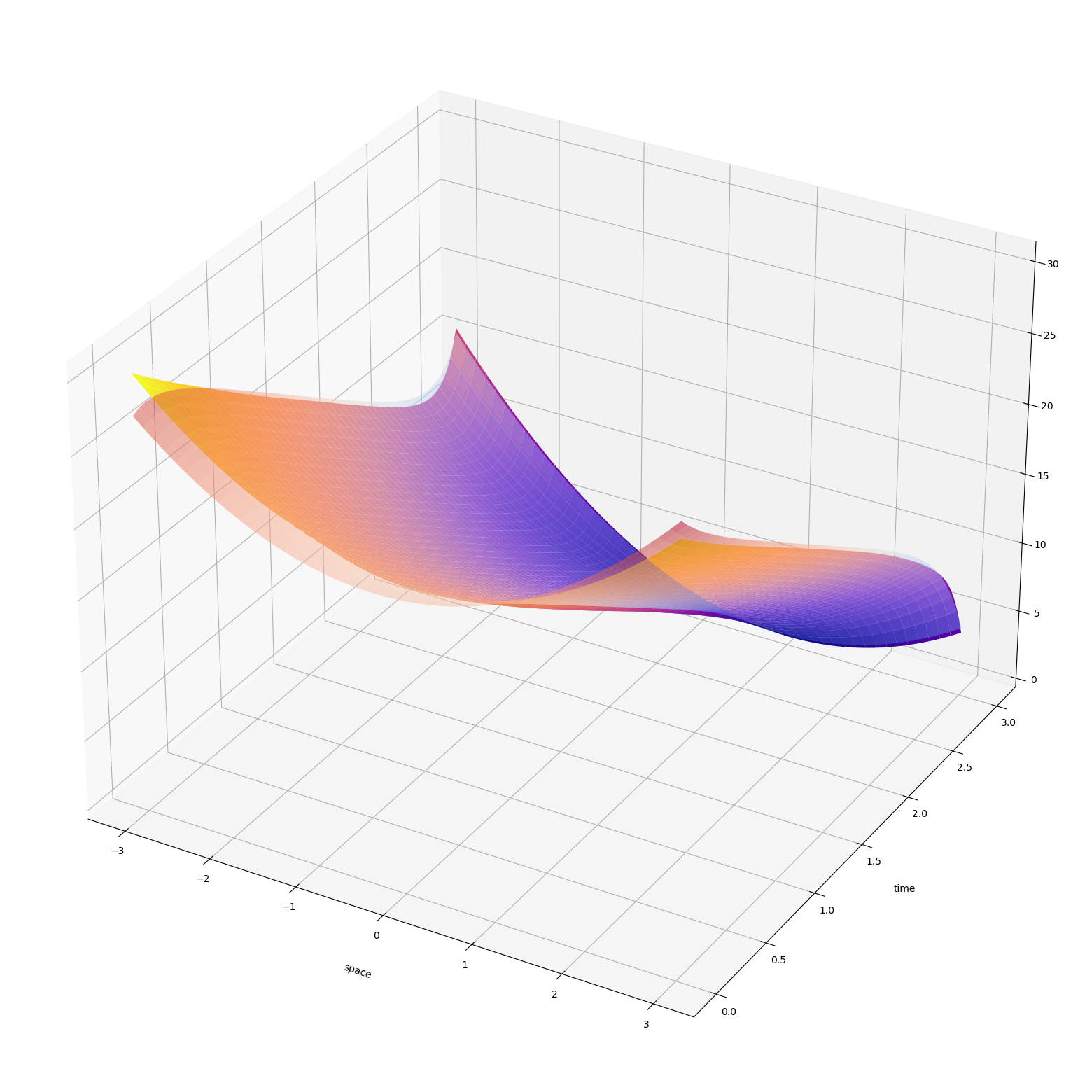}
    \end{subfigure}
    \bigskip
    \begin{subfigure}{0.7\textwidth}
        \includegraphics[width = \textwidth]{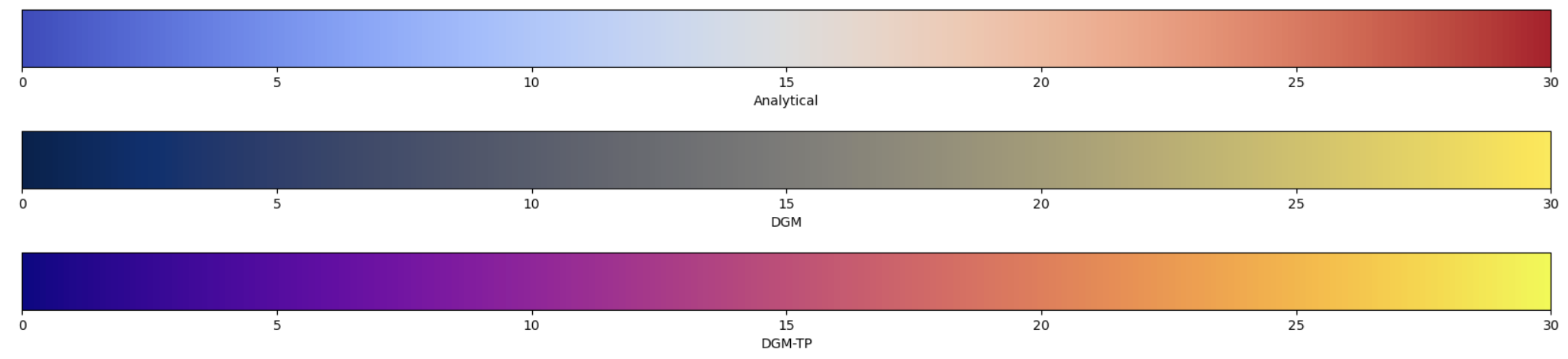}
    \end{subfigure}
    \caption{Plot of the value functions obtained through the analytical solution and the two DGM methods with respect to time in $[0,T]$ for $T=10$ and space in $[-3,3]$. The left pane shows the analytical solution against the DGM one, while the right pane shows the analytical solution against the DGM-TP $u$ one.}
    \label{fig:3D_value_function_comparison}
\end{figure}

\begin{figure}[hbt!]
    \centering
    \begin{subfigure}{0.45\textwidth}
        \includegraphics[width = \textwidth]{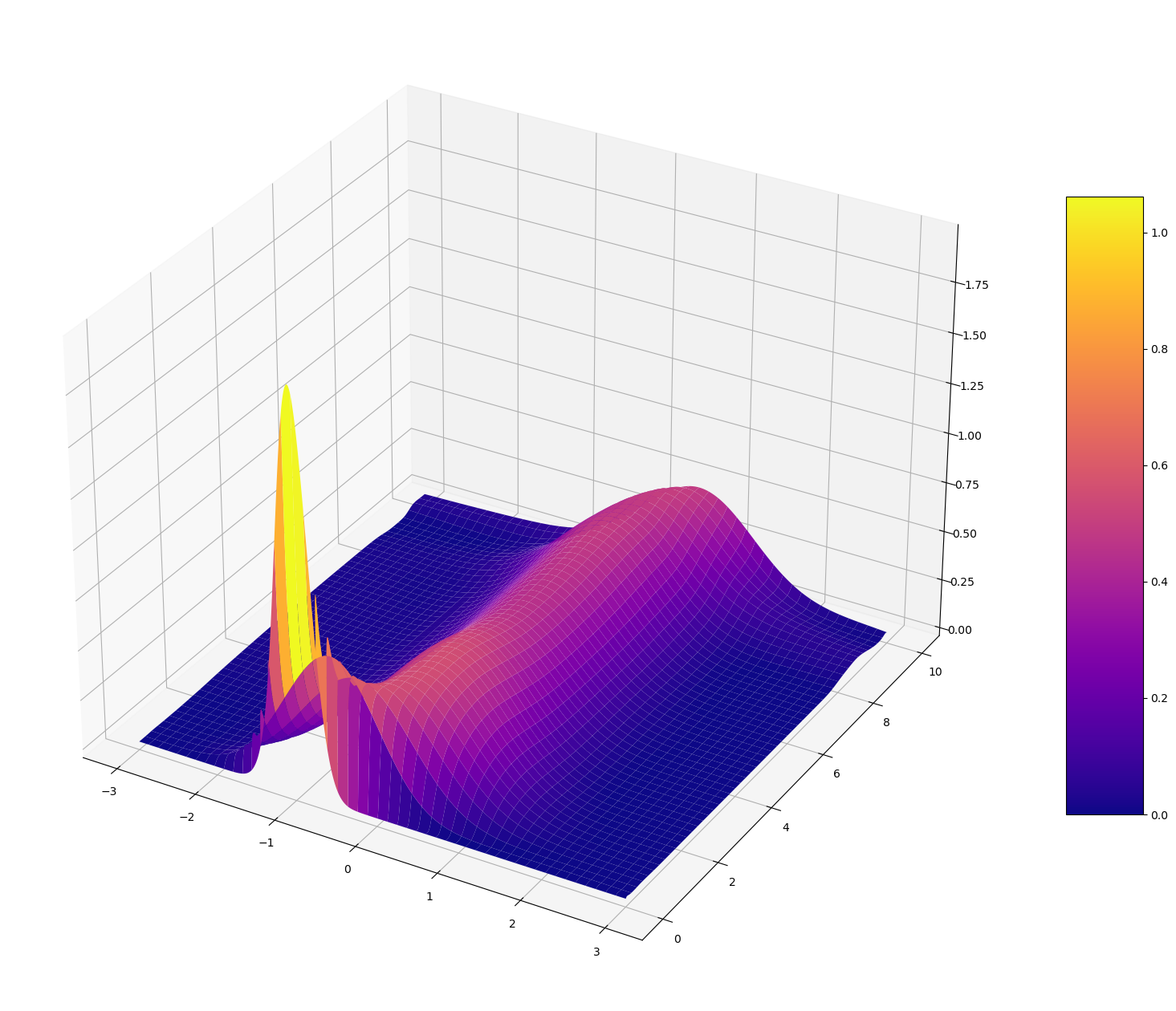}
    \end{subfigure}
    ~ 
    \begin{subfigure}{0.45\textwidth}
        \includegraphics[width = \textwidth]{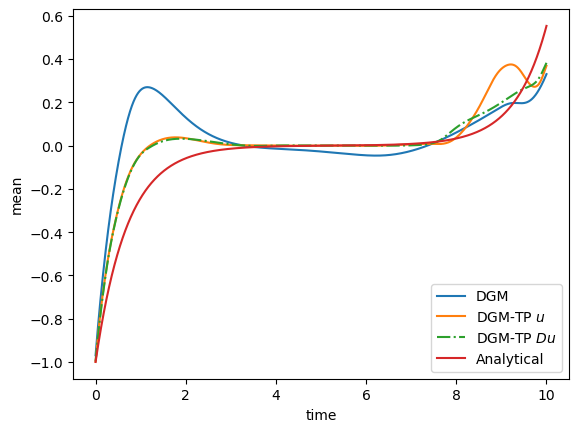}
    \end{subfigure}
    \caption{Evolution of the density function $m(t,x)$ from the DGM-TP $Du$ method (on the left) and of the average $\mu(t) = \int x m(t,dx)$ (on the right).}
    \label{fig:mean-lqmfg}
   \end{figure} 

The left pane of Figure \ref{fig:mean-lqmfg} gives the evolution of the density function for the DGM-TP-$u$ algorithm. We observe that the bell shape of the distribution is coherent with the linear-quadratic nature of the model with a Gaussian initial distribution. The shape is essentially the same for the baseline DGM and DGM-TP-$Du$ algorithms, up to the value of the mean. On the right pane, the evolution over time  of the average of the distribution $\mu = \int  x \mu(t,dx)$ is provided. The turnpike-accelerated methods lead to a process that is closer to the analytical solution with a middle phase around the steady-state that is slightly longer than the baseline DGM algorithm. 

We now compare the performance of the turnpike-accelerated DGM algorithm with the inclusion of the turnpike estimates involving on one hand $u$ and $\mu$ and on the other hand $Du$ and $\mu$. The relative errors are provided in Table \ref{tab:relative_error_lqmfg} and their graphs are represented in Figure \ref{fig:rel_err_lqmfg} : we notice that the turnpike-accelerated methods actually does improve the approximation on both the value function and on the distribution mean. There is however no significant discrepancy between both turnpike-accelerated versions: indeed, for the same penalty coefficient $C^{(TPK)}_{u}$, the slight variation comes from the actual value of the turnpike loss on $Du$ being smaller than that on $u$ and it therefore leads to giving a bit more importance to the distribution mean turnpike loss term in the first case than in the second. 

\begin{table}[hbt!]
    \centering
    \begin{tabular}{|c|c|c|c|c|}
    \hline 
        Included turnpike estimates & $C^{(TPK)}_{u/Du}$ & Relative Error on $u$ & $C^{(TPK)}_m$ & Relative Error on $\mu$  \\
        \hline 
        None & 0 & $1.113 \times 10^{-3}$ & 0 & $5.618 \times 10^{-2}$  \\
        \hline
       $u(t,x) - u(t,0)$ and $\mu(t,x)$ & 1 & $1.019 \times 10^{-4}$ & 0.1 & $3.042 \times 10^{-2}$\\
        \hline 
        $Du(t,x)$ and $\mu(t,x)$ & 1 & $1.334 \times 10^{-4}$ & 0.1 & $2.173 \times 10^{-2}$\\
        \hline 
    \end{tabular}
    \caption{Relative $L^2$-errors on the solutions for inclusion of different turnpike estimates into the loss function from the neural networks obtained after the last iteration $i = 400,000$. }
    \label{tab:relative_error_lqmfg}
\end{table}

\begin{figure}[hbt!]
    \centering
    \begin{subfigure}{0.45\textwidth}
        \includegraphics[width = \textwidth]{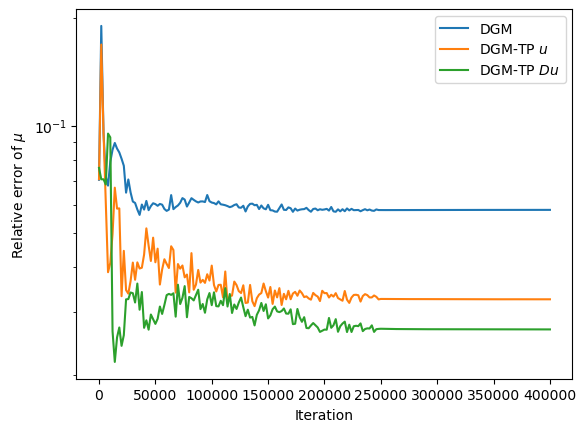}
    \end{subfigure}
    ~ 
    \begin{subfigure}{0.45\textwidth}
        \includegraphics[width = \textwidth]{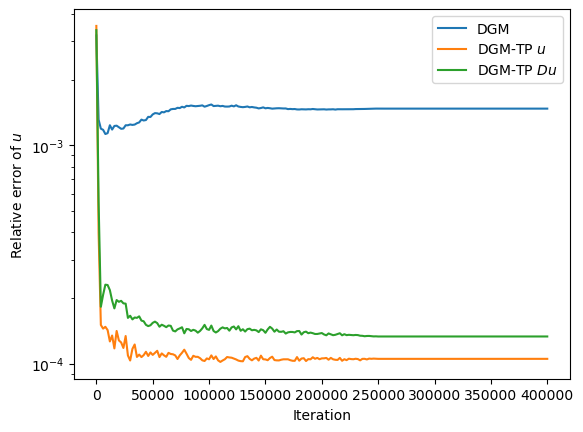}
    \end{subfigure}
    \caption{Plot of the relative errors for the distribution mean $\mu(t) = \int_{Q} x m(t,dx)$ (left) and the value function (right).}
    \label{fig:rel_err_lqmfg}
   \end{figure} 

\newpage

%%%%%%%%%%%%%%%%%%%%%%%%%%%%%%%%%%%%%%%%%%%%%%%%%%%%%%%%%
\section{\textbf{Conclusions, Limitations and Possible Extensions} }
\label{section:5}

The thrust of the paper is to propose an original algorithm to compute numerically Mean Field Game (MFG) equilibria over long time horizons for a specific class of models. Time-dependent equilibria are notoriously difficult to approximate when the time horizon is large, and even the best known numerical methods can experience difficulties converging, or take too much time when they do converge. Our strategy to compute time-dependent equilibria over long time horizons is to assume that the coefficients of the model are time-independent, and that the solution of the corresponding stationary (i.e. time-independent) ergodic version of the model can be computed efficiently. According to the turnpike property, we shall compute the solution to the time dependent problem only near the time origin and the terminal time, and use the ergodic (stationary) solution for the bulk of the remaining time. In layman terms, we turn the turnpike phenomenon into a practical numerical recipe to compute the equilibrium over long time intervals. 

We propose to modify an existing variational method to incorporate the knowledge of the turnpike property for some Mean Field Game models. We implement a version of the Deep Garlekin Method (DGM) to compute the distribution function and the value function of the standard coupled MFG PDE system. In our formulation of the algorithm, both functions are penalized if they fail to satisfy their own expected \emph{turnpike estimates}. We demonstrate the rationale for our approach by computing numerically the solutions of models for which the turnpike phenomenon was argued theoretically, and highlight the core elements of our proposal.

This first class of MFG models has local interactions. We quantify numerically the superiority of our \emph{turnpike-accelerated DGM} over the baseline DGM and an implementation of the standard Finite Difference Method (FDM) used as baseline since no explicit solution is known for the model in question.
The second class of models on which we test the performance of our numerical algorithm comprise Linear Quadratic (LQ) MFGs with non-local interactions. While those models have been extensively studied, mostly because \emph{semi-explicit solutions} are available for finite time-horizons (even if it is at the cost of the solution of unstable Riccati equations and ODE systems involving the statistical moments of the mean field distribution), they have not been studied for the turnpike property. We derive rigorously the turnpike estimates holding for this class of model, and we proceed to show the superiority of the \emph{turnpike-accelerated DGM} over the baseline DGM method and other standard algorithms.

Given the versatility of the DGM algorithm when it comes to tackling high-dimensional problems, the question of the scalability of the suggested version to higher dimensions is of particular interest. The experiments conducted on Mean Field Games suggest to revisit some PDE applications, since turnpike properties are also available for stochastic control problems. Moreover, the turnpike-accelerated method suggests integrating other types of properties (if available) in the Deep Galerkin Method loss to potentially enhance the performance of the numerical scheme. Finally, the newly derived turnpike estimates for the LQ models open the possibility of exploring the application of transfer learning for the training of Mean Field Games with the turnpike property. Transfer learning refers to the warm-start of the weights of the neural networks for a specific task, using the solutions trained for a related but usually simpler problem.  The idea would be to first train the neural networks to solve the ergodic MFG PDE system and leverage this knowledge to accelerate the training of the decoupling field of the Forward-Backward Stochastic Differential Equation (FBSDE) of the finite horizon problem obtained through the Pontryagin Stochastic Maximum Principle. This would provide an initialization of the Shooting Method \textit{à la Sannikov} using the PDE approach of the ergodic MFG system.

\printbibliography

@book{CarmonaDelarue_book_I,
    AUTHOR = {Carmona, Ren\'{e} and Delarue, Fran\c{c}ois},
     TITLE = {Probabilistic theory of mean field games with applications.
              {I}},
    SERIES = {Probability Theory and Stochastic Modelling},
    VOLUME = {83},
      NOTE = {Mean field FBSDEs, control, and games},
 PUBLISHER = {Springer, Cham},
      YEAR = {2018},
     PAGES = {xxv+713},
      ISBN = {978-3-319-56437-1; 978-3-319-58920-6},
   MRCLASS = {60-02 (35R60 49N70 49N90 60H15 60H30 91A15 93E20)},
  MRNUMBER = {3752669},
MRREVIEWER = {Vassili N. Kolokol\cprime tsov},
}

@book{dorfman_linear_1987,
	title = {Linear Programming and Economic Analysis},
	isbn = {9780486654911},
	pagetotal = {546},
	publisher = {Courier Corporation},
	author = {Dorfman, Robert and Samuelson, Paul Anthony and Solow, Robert M.},
	date = {1987-01-01},
	langid = {english},
}

@article{cirant_long_2021,
	title = {Long time behavior and turnpike solutions in mildly non-monotone mean field games},
	volume = {27},
	url = {https://www.esaim-cocv.org/articles/cocv/abs/2021/02/cocv200241/cocv200241.html},
	pages = {86},
	journaltitle = {{ESAIM}: Control, Optimisation and Calculus of Variations},
	author = {Cirant, Marco and Porretta, Alessio},
	date = {2021},
}

@article{porretta_turnpike_2018,
	title = {On the turnpike property for mean field games},
	volume = {3},
	url = {https://art.torvergata.it/handle/2108/231063},
	pages = {285--312},
	number = {2},
	journaltitle = {Minimax Theory and its Applications},
	author = {Porretta, Alessio},
	date = {2018},
}

@article{cardaliaguet_long_2019,
	title = {Long time behavior of the master equation in mean field game theory},
	volume = {12},
	url = {https://msp.org/apde/2019/12-6/p01.xhtml},
	pages = {1397--1453},
	number = {6},
	journaltitle = {Analysis \& {PDE}},
	author = {Cardaliaguet, Pierre and Porretta, Alessio},
	date = {2019},
}

@article{porretta_long_2013,
	title = {Long Time versus Steady State Optimal Control},
	volume = {51},
	issn = {0363-0129, 1095-7138},
	url = {http://epubs.siam.org/doi/10.1137/130907239},
	doi = {10.1137/130907239},
	pages = {4242--4273},
	number = {6},
	journaltitle = {{SIAM} Journal on Control and Optimization},
	shortjournal = {{SIAM} J. Control Optim.},
	author = {Porretta, Alessio and Zuazua, Enrique},
	date = {2013-01},
	langid = {english},
}

@article{carmona_convergence_2021,
	title = {Convergence Analysis of Machine Learning Algorithms for the Numerical Solution of Mean Field Control and Games I: The Ergodic Case},
	volume = {59},
	issn = {0036-1429},
	url = {https://epubs.siam.org/doi/abs/10.1137/19M1274377},
	doi = {10.1137/19M1274377},
	shorttitle = {Convergence Analysis of Machine Learning Algorithms for the Numerical Solution of Mean Field Control and Games I},
	pages = {1455--1485},
	number = {3},
	journaltitle = {{SIAM} Journal on Numerical Analysis},
	shortjournal = {{SIAM} J. Numer. Anal.},
	author = {Carmona, Ren\'e and Lauri\`ere, Mathieu},
	date = {2021-01},
}

@article{porretta_long_2012,
	title = {Long time average of mean field games},
	volume = {7},
	doi = {10.3934/nhm.2012.7.279},
	pages = {279--301},
	journaltitle = {Networks and Heterogeneous Media},
	shortjournal = {Networks and Heterogeneous Media},
	author = {Porretta, Alessio and Cardaliaguet, Pierre and Lasry, Jean-Michel and Lions, Pierre-Louis},
	date = {2012-06-01},
}

@book{achdou_mean_2020,
	location = {Cham},
	title = {Mean Field Games: Cetraro, Italy 2019},
	volume = {2281},
	isbn = {9783030598365 9783030598372},
	url = {http://link.springer.com/10.1007/978-3-030-59837-2},
	series = {Lecture Notes in Mathematics},
	shorttitle = {Mean Field Games},
	publisher = {Springer International Publishing},
	author = {Achdou, Yves and Cardaliaguet, Pierre and Delarue, Fran\c{c}ois and Porretta, Alessio and Santambrogio, Filippo},
	editor = {Cardaliaguet, Pierre and Porretta, Alessio},
	date = {2020},
	langid = {english},
	doi = {10.1007/978-3-030-59837-2},
}

@article{gomes_discrete_2010,
	title = {Discrete time, finite state space mean field games},
	volume = {93},
	issn = {0021-7824},
	url = {https://www.sciencedirect.com/science/article/pii/S002178240900138X},
	doi = {10.1016/j.matpur.2009.10.010},
	pages = {308--328},
	number = {3},
	journaltitle = {Journal de Math\'ematiques Pures et Appliqu\'ees},
	shortjournal = {Journal de Math\'ematiques Pures et Appliqu\'ees},
	author = {Gomes, Diogo A. and Mohr, Joana and Souza, Rafael Rig\~ao},
	date = {2010-03-01},
}

@article{ferreira_convergence_2014,
	title = {On the convergence of finite state mean-field games through Gamma-convergence},
	volume = {418},
	issn = {0022-247X},
	url = {https://www.sciencedirect.com/science/article/pii/S0022247X14001735},
	doi = {10.1016/j.jmaa.2014.02.044},
	pages = {211--230},
	number = {1},
	journaltitle = {Journal of Mathematical Analysis and Applications},
	shortjournal = {Journal of Mathematical Analysis and Applications},
	author = {Ferreira, Rita and Gomes, Diogo A.},
	date = {2014-10-01},
}

@article{lasry_jeux_2006_b,
	title = {Jeux \`a champ moyen. {II} – Horizon fini et contr\^ole optimal},
	volume = {343},
	issn = {1631-073X},
	url = {https://www.sciencedirect.com/science/article/pii/S1631073X06003670},
	doi = {10.1016/j.crma.2006.09.018},
	pages = {679--684},
	number = {10},
	journaltitle = {Comptes Rendus Math\'ematiques},
	shortjournal = {Comptes Rendus Mathematique},
	author = {Lasry, Jean-Michel and Lions, Pierre-Louis},
	date = {2006-11-15},
}

@article{lasry_jeux_2006_a,
	title = {Jeux \`a champ moyen. I – Le cas stationnaire},
	volume = {343},
	issn = {1631-073X},
	url = {https://www.sciencedirect.com/science/article/pii/S1631073X06003682},
	doi = {10.1016/j.crma.2006.09.019},
	pages = {619--625},
	number = {9},
	journaltitle = {Comptes Rendus Mathematique},
	shortjournal = {Comptes Rendus Mathematique},
	author = {Lasry, Jean-Michel and Lions, Pierre-Louis},
	date = {2006-11-01},
}

@article{trelat_steady-state_2018,
	title = {Steady-State and Periodic Exponential Turnpike Property for Optimal Control Problems in Hilbert Spaces},
	volume = {56},
	issn = {0363-0129},
	url = {https://epubs.siam.org/doi/10.1137/16M1097638},
	doi = {10.1137/16M1097638},
	pages = {1222--1252},
	number = {2},
	journaltitle = {{SIAM} Journal on Control and Optimization},
	shortjournal = {{SIAM} J. Control Optim.},
	author = {Tr\'elat, Emmanuel and Zhang, Can and Zuazua, Enrique},
	date = {2018-01},
}

@article{trelat_turnpike_2015,
	title = {The turnpike property in finite-dimensional nonlinear optimal control},
	volume = {258},
	issn = {0022-0396},
	url = {https://www.sciencedirect.com/science/article/pii/S0022039614003568},
	doi = {10.1016/j.jde.2014.09.005},
	pages = {81--114},
	number = {1},
	journaltitle = {Journal of Differential Equations},
	shortjournal = {Journal of Differential Equations},
	author = {Tr\'elat, Emmanuel and Zuazua, Enrique},
	date = {2015-01-01},
}

@article{Bardi_explicit_sol,
title = {Explicit solutions of some linear-quadratic mean field games},
journal = {Networks and Heterogeneous Media},
volume = {7},
number = {2},
pages = {243-261},
year = {2012},
issn = {1556-1801},
doi = {10.3934/nhm.2012.7.243},
url = {https://www.aimsciences.org/article/id/18044780-7a35-47d0-be36-f3efe2ba67bc},
author = {Martino Bardi},
}

@article{OE_anon_id,
author = {Ren\'e Carmona and Claire Zeng},
title = {Optimal Execution with Identity Optionality},
journal = {Applied Mathematical Finance},
volume = {29},
number = {4},
pages = {261-287},
year = {2022},
publisher = {Routledge},
doi = {10.1080/1350486X.2023.2193343},
URL = {https://doi.org/10.1080/1350486X.2023.2193343},
eprint = {https://doi.org/10.1080/1350486X.2023.2193343}
}

@article{cardialaguet-nonlocal,
author = {Cardaliaguet, P. and Lasry, J.-M. and Lions, P.-L. and Porretta, A.},
title = {Long Time Average of Mean Field Games with a Nonlocal Coupling},
journal = {SIAM Journal on Control and Optimization},
volume = {51},
number = {5},
pages = {3558-3591},
year = {2013},
doi = {10.1137/120904184},
URL = {https://doi.org/10.1137/120904184},
eprint = {https://doi.org/10.1137/120904184}
}

@article{Kuramoto,
author = {Annalisa Cesaroni and Marco Cirant},
title = {Stationary equilibria and their stability in a Kuramoto MFG with strong interaction},
journal = {Communications in Partial Differential Equations},
volume = {0},
number = {0},
pages = {1-27},
year = {2024},
publisher = {Taylor & Francis},
doi = {10.1080/03605302.2023.2300824},
URL = { https://doi.org/10.1080/03605302.2023.2300824},
eprint = {https://doi.org/10.1080/03605302.2023.2300824}

}

@article{SIRIGNANO20181339,
title = {DGM: A deep learning algorithm for solving partial differential equations},
journal = {Journal of Computational Physics},
volume = {375},
pages = {1339-1364},
year = {2018},
issn = {0021-9991},
doi = {https://doi.org/10.1016/j.jcp.2018.08.029},
url = {https://www.sciencedirect.com/science/article/pii/S0021999118305527},
author = {Justin Sirignano and Konstantinos Spiliopoulos}
}

@article{ANDERSON1982313,
title = {Reverse-time diffusion equation models},
journal = {Stochastic Processes and their Applications},
volume = {12},
number = {3},
pages = {313-326},
year = {1982},
issn = {0304-4149},
doi = {https://doi.org/10.1016/0304-4149(82)90051-5},
url = {https://www.sciencedirect.com/science/article/pii/0304414982900515},
author = {Brian D.O. Anderson}
}

@inproceedings{song2021scorebased,
title={Score-Based Generative Modeling through Stochastic Differential Equations},
author={Yang Song and Jascha Sohl-Dickstein and Diederik P Kingma and Abhishek Kumar and Stefano Ermon and Ben Poole},
booktitle={International Conference on Learning Representations},
year={2021},
url={https://openreview.net/forum?id=PxTIG12RRHS}
}

@article{ARAPOSTATHIS2017205,
title = {On solutions of mean field games with ergodic cost},
journal = {Journal de Mathématiques Pures et Appliquées},
volume = {107},
number = {2},
pages = {205-251},
year = {2017},
issn = {0021-7824},
doi = {https://doi.org/10.1016/j.matpur.2016.06.004},
url = {https://www.sciencedirect.com/science/article/pii/S0021782416300666},
author = {Ari Arapostathis and Anup Biswas and Johnson Carroll}
}

\begin{appendix}

%%%%%%%%%%%%%%%%%%%%%%%%%%%%%%%%%%%%%%%%%%
\section{\textbf{Finite Difference Discretization} }
\label{appendix:A}
For the sake of completeness, we describe the finite difference method used to compute the finite horizon solution in the local coupling model from Section\ref{section:3}. 

\subsection{Notations}

Let $N_T$ and $N_h$ denote two positive integers for a discretization on $[0,T]$ and $[0,1]$ with respectively $(N_T+1)$ and $(N_h+1)$ points in time and space. Let $\Delta t = T/N_T$ and $h = 1/N_h$ and $t_n = n \times \Delta t$ and $x_i = i \times h$ for $n \in \{0, \dots, N_T\}$ and $ i \in \{0, \dots, N_h\}$. 
We approximate $u$ and $m$ by the vectors $U$ and $M$ such that $u(t_n, x_i) \simeq U^n_i$ and $m(t_n, x_i) \simeq M^n_i$. Let $F$ be a generic matrix in $\mathcal{M}_{N_T+1 \times N_h +1}$. We consider the following finite difference operators: 
\begin{align*}
    & F_{N_h + 1}  \equiv F_0, & \text{(Convention)} \\
    &(D_tF)^n = \frac{1}{\Delta t} (F^{n+1} - F^n), & n \in \{0, \dots, N_T-1\}, \\
    & (DF)_i = \frac{1}{h} (F_{i+1} - F_i), & i \in \{0, \dots, N_h\}, \\
      & (\Delta_h F)_i = -\frac{1}{h^2} (2 F_i - F_{i+1} - F_{i-1}), & i \in \{0, \dots, N_h\},  \\  
     & [\nabla_h F]_i = \Big( (DF)_i ,  (DF)_{i-1} \Big)^{\dagger}, & i \in \{0, \dots, N_h\}.
\end{align*}

\begin{algorithm}
\caption{Fixed point iterations of the FDM algorithm}\label{alg:cap}
\begin{algorithmic}
\Require Initial guess $(\tilde{M}, \tilde{V})$; damping $\delta$, number of iterations $K$
\Ensure Approximation of $(\hat{M}, \hat{V})$ solving the finite difference system. \\
\textbf{Initialize}
\State $M^{(0)}  \gets \tilde{M}$
\State $V^{(0)} \gets \tilde{V}$
\For{$k = 0, 1, \dots, K-1$} \\
    \State Let $V^{(k+1)}$ be the solution of: 
    \State \begin{align*}
        \begin{cases}
            & - (D_t V_i)^n - \kappa (\Delta_h V^{n})_i + \tilde{H}(x_i, [ \nabla_h V^{n}]_i) = F_h[M^{(k), n+1}]_i, \, 0 \leq n \leq N_T-1,  \\
& V^{N_{T}}_i = \phi(x_i).
        \end{cases}
    \end{align*}  \Comment{Since the HJB equation is non-linear, we use a Newton method to linearize its discrete version.}
\State Let $M^{(k+1)}$ be the solution of: 
   \begin{align*}
        \begin{cases}
            & - (D_t M_i)^n - \kappa (\Delta_h M^{n})_i - \mathcal{T}_i (V^{(k+1), n+1}, M^{n}) = 0 ,\, 0 \leq n \leq N_T-1,  \\
& M^{N_{T_1}}_i = \psi(x_i). 
        \end{cases}
    \end{align*}  
    Let $\tilde{M}^{(k+1)} = \delta(k) \tilde{M}^{(k)} + (1-\delta(k)) M^{(k+1)}$. 
\EndFor \\
\Return $(M^{(K)}, V^{(K)})$
\end{algorithmic}
\end{algorithm}

\subsection{Discretization on $[0, T]$}

We respectively denote the initial and terminal conditions as $\phi$ and $\chi$. Then the PDE system becomes: 

 \begin{subequations}
    \begin{empheq}[left={\empheqlbrace}]{align}
& - \partial_t v - \kappa \Delta u + H(x, Du) =  F(x, m) &\text{ in } (0, T) \times \mathbb{T}^d, \\
& \partial_t \mu  - \kappa \Delta \mu - \operatorname{div}(\mu H_p(x, Dv)) = 0 &\text{ in } (0, T) \times \mathbb{T}^d,  \\
&v(T,x) = \phi(x), \, \mu(0,x) = \chi(x), \int m(x)dx = 1.
    \end{empheq}
    \end{subequations}

\subsubsection{Discrete HJB Equation}
 \begin{subequations}
    \begin{empheq}[left={(dHJB) \empheqlbrace}]{align*}
& - (D_t V_i)^n - \kappa (\Delta_h V^{n})_i + \tilde{H}(x_i, [ \nabla_h V^{n}]_i) = F_h[M^{n+1}]_i, & 0 \leq i \leq N_h, 0 \leq n \leq N_T-1,  \\
& V^{T}_i = \phi(x_i), & \text{initialization} \\
& V^{n}_0 = V^{n}_{N_h} , & \text{periodic conditions}
    \end{empheq}
    \end{subequations}

Since $H(x,p) = \frac{1}{2} \lvert p \rvert^2$, we consider  $
    \tilde{H} (x,p_1, p_2) = \frac{1}{2} \lvert P_K(p_1, p_2) \rvert^2 $
where $P_K$ is the projection on $ K = \mathbb{R}_{-} \times \mathbb{R}_{-}$. 

\subsubsection{Discrete KFP Equation}
 \begin{subequations}
    \begin{empheq}[left={(dKFP):= \empheqlbrace}]{align*}
& - (D_t M_i)^n - \kappa (\Delta_h M^{n+1})_i - \mathcal{T}_i (V^{n+1}, M^{n+1}) = 0 , & 0 \leq i \leq N_h, 0 \leq n \leq N_T-1,  \\
& M^{0}_i = \chi(x_i), & \text{initialization} \\
& M^{n+1}_i \geq 0 , & \text{non-negativity} \\ 
& \frac{1}{h} \sum_{i=0}^{N_h} M^{n}_i = 1 , &0 \leq n \leq N_T-1, \, \text{normalization} \\ 
& M^{n}_0 = M^{n}_{N_h} , & \text{periodicity}
\end{empheq}
\end{subequations}
where 
\begin{align*}
    \mathcal{T}_i (U, M)= & \frac{1}{h} \Big( M_i \tilde{H}_{p_1}(x_i, [\nabla_h U]_i) -  M_{i-1} \tilde{H}_{p_1}(x_{i-1}, [\nabla_h U]_{i-1}) \\
    & \qquad + M_{i+1} \tilde{H}_{p_2}(x_{i+1}, [\nabla_h U]_{i+1}) -  M_i \tilde{H}_{p_2}(x_i, [\nabla_h U]_i)\Big)
\end{align*}

\section{\textbf{DGM for an Ergodic MFG}} 
\label{appendix:B}

The PDE system is reformulated as an optimization problem with an objective function with the triple $(\lambda, u, m)$ being the control. The loss function actually corresponds to the sum of the PDE residuals and some additional boundary and normalization conditions. We approximate the functions $u(x)$ and $m(x)$ as two deep neural networks $u_{\theta} := \phi(\cdot;\theta)$ and  $m_{\eta} := \chi(\cdot;\eta)$ where $\theta \in \mathbb{R}^K$ and $\eta \in \mathbb{R}^K$ are the respective neural networks parameters to be learnt. We also consider a constant neural network $\lambda_{\zeta} = \zeta$ with $\zeta \in \mathbb{R}$ for the ergodic constant.  We denote for $x = (x_1, x_2, \dots, x_d) \in \mathbb{R}^d$ and $ z \in \mathbb{R}$, $(x^{-k}, z) := (x_1, \dots, x_{k-1}, z, x_{k+1}, \dots, x_d)$. We then consider the following loss function : 
\begin{align*}
    \mathcal{L}(\lambda, u, m) = & \, \, C^{(HJB)} \mathcal{L}^{(HJB)}(\lambda, u, m) +  C^{(KFP)}  \mathcal{L}^{(KFP)} (\lambda, u, m) \\ 
    & + C^{(norm)} \mathcal{L}^{(norm)}(\lambda, u, m) + C^{(period)} 
 \mathcal{L}^{(period)}(\lambda, u, m) 
\end{align*}
where, using the notation $F[m]=  F(x , m(x))$,  
\begin{align*}
    & \mathcal{L}^{(HJB)}(\lambda, u, m)  = \Big\lVert \lambda - \frac{1}{2} \Delta u + \frac{1}{2} \lvert \nabla u \rvert^2 - F[m] \Big\rVert_{L^2(\mathbb{T}^d)}^2, \\
    & \mathcal{L}^{(KFP)}(\lambda, u, m)  = \Big\lVert- \frac{1}{2} \Delta m - \operatorname{div} (m \nabla u) \Big\rVert_{L^2(\mathbb{T}^d)}^2, \\
    &  \mathcal{L}^{(norm)}(\lambda, u, m) = \Big \lvert \int_{\mathbb{T}^d} u(x) dx \Big\rvert + \Big\lvert \int_{\mathbb{T}^d} m(t,x) dx - 1 \Big\rvert, \\
    & \mathcal{L}^{(period)}(\lambda, u, m) =  \sum_{k=1}^d \lvert u((x^{-k}, 0)) - u((x^{-k}, 1) ) \rvert^2 + \sum_{k=1}^d  \lvert m((x^{-k}, 0)) - m((x^{-k}, 1) ) \rvert^2 . 
\end{align*}
and $ C^{(HJB)}$, $C^{(KFP)}$, $C^{(norm)}$, and $C^{(period)}$ are positive coefficients controlling the importance given to each component. The objective is to find the right set of parameters $\theta$, $\eta$, so that the neural networks $\phi$, $\chi$, $\lambda_{\omega}$ minimize the loss functional $\mathcal{L}(\lambda, u , m)$. To that end, we use a Monte Carlo method to approximate the $L^2$ norms and compute the above residual terms as: 

\begin{align*}
    & L^{(HJB)}(\lambda_{\zeta}, u_{\theta}, m_{\eta})  = \frac{1}{M} \sum_{k=1}^{M}  \Big\lvert \lambda_{\zeta} - \frac{1}{2} \Delta u(x_k) + \frac{1}{2} \lvert \nabla u(x_k) \rvert^2 - F[m][x_k] \Big\rvert^2, \\
    & L^{(KFP)}(\lambda_{\zeta}, u_{\theta}, m_{\eta})  = \frac{1}{M} \sum_{k=1}^{M}  \Big\lvert- \frac{1}{2} \Delta m(x_k) - \operatorname{div} \Big(m(x_k) \nabla u(x_k) \Big)  \Big\rvert^2, \\
    & L^{(norm)}(\lambda_{\zeta}, u_{\theta}, m_{\eta})  = \Big \lvert \frac{1}{M} \sum_{k=1}^M u(x_k) \Big\rvert + \Big\lvert \frac{1}{M} \sum_{k=1}^M  m(x_k) - 1 \Big\rvert, \\
    & L^{(period)}(\lambda_{\zeta}, u_{\theta}, m_{\eta})  =  \frac{1}{M} \sum_{l=1}^M \sum_{k=1}^d \lvert u((x_{l}^{-k}, 0)) - u((x_{l}^{-k}, 1) ) \rvert^2 + \frac{1}{M} \sum_{l=1}^M \sum_{k=1}^d  \lvert m((x_l^{-k}, 0)) - m((x_l^{-k}, 1) ) \rvert^2. 
\end{align*}

\begin{figure}[hbt!]
    \centering
    \begin{subfigure}[t]{0.45\textwidth}
        \centering
        \includegraphics[width = 6cm]{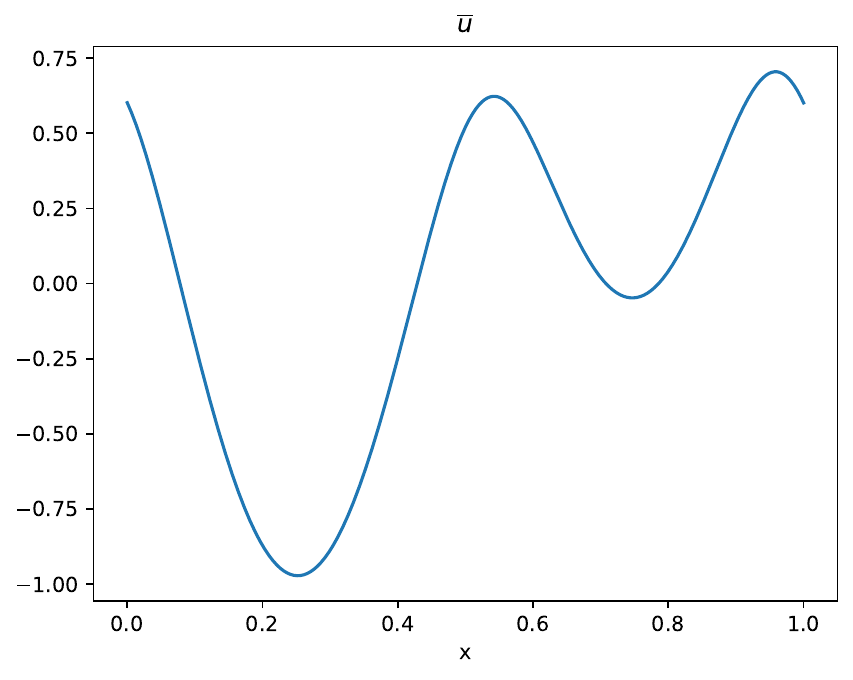}
        \caption{Plot of the normalized value function $\overline{u}$ with $\int_{\mathbb{T}} \overline{u}(x) dx = 0$}
    \end{subfigure} 
    ~
    \begin{subfigure}[t]{0.45\textwidth}
        \centering
          \includegraphics[width = 6cm]{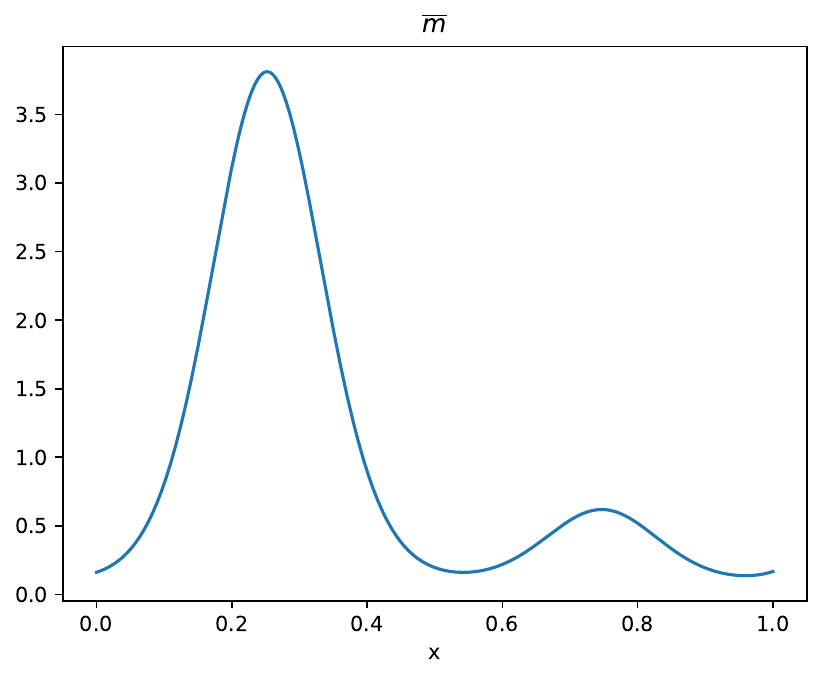}
    \caption{Plot of the density function $\overline{m}$}
    \end{subfigure}
    \caption{Stationary solution of the ergodic MFG}
\end{figure}

\end{appendix}

\end{document}